\documentclass[12pt,letterpaper,oneside,reqno]{amsart}

\usepackage[utf8]{inputenc}
\usepackage[english]{babel}
\numberwithin{equation}{section}

\usepackage{hyperref}

\usepackage[margin=1in]{geometry}
\usepackage{mathtools}
\usepackage{enumitem}

\usepackage{amsthm}
\usepackage{amssymb}
\usepackage{amsfonts}
\usepackage[mathcal]{eucal}
\usepackage{bbm}
\usepackage{latexsym}
\usepackage{mathrsfs}
\usepackage{stmaryrd}

\makeatletter
\DeclareFontFamily{OMX}{MnSymbolE}{}
\DeclareSymbolFont{MnLargeSymbols}{OMX}{MnSymbolE}{m}{n}
\SetSymbolFont{MnLargeSymbols}{bold}{OMX}{MnSymbolE}{b}{n}
\DeclareFontShape{OMX}{MnSymbolE}{m}{n}{
    <-6>  MnSymbolE5
   <6-7>  MnSymbolE6
   <7-8>  MnSymbolE7
   <8-9>  MnSymbolE8
   <9-10> MnSymbolE9
  <10-12> MnSymbolE10
  <12->   MnSymbolE12
}{}
\DeclareFontShape{OMX}{MnSymbolE}{b}{n}{
    <-6>  MnSymbolE-Bold5
   <6-7>  MnSymbolE-Bold6
   <7-8>  MnSymbolE-Bold7
   <8-9>  MnSymbolE-Bold8
   <9-10> MnSymbolE-Bold9
  <10-12> MnSymbolE-Bold10
  <12->   MnSymbolE-Bold12
}{}

\let\llangle\@undefined
\let\rrangle\@undefined
\DeclareMathDelimiter{\llangle}{\mathopen}%
{MnLargeSymbols}{'164}{MnLargeSymbols}{'164}

\DeclareMathDelimiter{\rrangle}{\mathclose}%
{MnLargeSymbols}{'171}{MnLargeSymbols}{'171}

\theoremstyle{plain}
\newtheorem*{MET}{Mean Ergodic Theorem (MET)}
\newtheorem{theorem}{Theorem}
\newtheorem{proposition}{Proposition}[section]
\newtheorem{lemma}[proposition]{Lemma}
\newtheorem{corollary}[proposition]{Corollary}

\theoremstyle{definition}

\newtheorem{definition}[proposition]{Definition}
\newtheorem{convention}[proposition]{Convention}
\newtheorem{notation}[proposition]{Notation}

\theoremstyle{remark}
\newtheorem{remark}[proposition]{Remark}
\newtheorem{remarks}[proposition]{Remarks}

\numberwithin{equation}{section}

\DeclareMathOperator{\AV}{AV}
\DeclareMathOperator{\card}{card}

\DeclareMathOperator{\osc}{osc}
\DeclareMathOperator{\oprod}{\prescript{\mathrm{op}}{}\prod}
\DeclareMathOperator{\spr}{spr}
\DeclareMathOperator{\Rand}{Rand}

\DeclareMathOperator{\Struct}{Struct}

\DeclareMathOperator{\Th}{Th}
\DeclareMathOperator{\tp}{tp}

\DeclareMathOperator{\wneg}{\mathrel{\ooalign{\hss$\neg$\hss\cr\kern0.2ex\raise1.0ex\hbox{\scalebox{0.7}{w}}}}}

\renewcommand{\restriction}{\mathord{\upharpoonright}}

\newcommand{\ab}{\ensuremath{a_{\bullet}}}
\newcommand{\AVb}{\AV_{\bullet}}
\newcommand{\cA}{\mathcal{A}}

\newcommand{\Amu}{\ensuremath{\llb\mathcal{\cA}\rrb_{\mu}}}
\newcommand{\AO}{\cA_{\Omega}}
\newcommand{\AOM}{\AO^{\cM}}
\newcommand{\AZ}{\cA_{\ZZ}}

\newcommand{\BB}{\mathfrak{B}}

\newcommand{\Bone}{\ensuremath{\mathbbm{I}}}

\newcommand{\cF}{\mathcal{F}}
\newcommand{\Fb}{\ensuremath{\cF_{\bullet}}}
\newcommand{\GG}{\ensuremath{\mathbb{G}}}
\newcommand{\GAB}{\ensuremath{\langle A,B\rangle}}
\newcommand{\Gab}{\ensuremath{\langle a,b\rangle}}
\newcommand{\GD}{\ensuremath{\langle\Delta^{\!\!\circ} T\rangle_{\!\bullet}}}
\newcommand{\ccAB}{\ensuremath{\cc{\GAB}}}
\newcommand{\GZ}{\ensuremath{G^{\ZZ}}}
\newcommand{\GZD}{\ensuremath{\langle\Delta^{\!\!\circ} T\rangle}}
\newcommand{\cH}{\ensuremath{\mathcal{H}}}
\newcommand{\bin}[2]{\ensuremath{\llb{#1}\boldsymbol{\in}{#2}\rrb}}
\newcommand{\cL}{\ensuremath{\mathcal{L}}}
\newcommand{\sL}{\ensuremath{\mathscr{L}}}
\newcommand{\cM}{\ensuremath{\mathscr{M}}}
\newcommand{\Nab}{\ensuremath{\nabla^{\bullet}}}
\newcommand{\cN}{\mathcal{N}}
\newcommand{\cU}{\mathcal{U}}
\newcommand{\cX}{\mathcal{X}}
\newcommand{\cY}{\mathcal{Y}}
\newcommand{\cZ}{\mathcal{Z}}
\newcommand{\degL}{\deg_{\mathrm{L}}}
\newcommand{\DD}{\ensuremath{\mathbb{D}}}
\newcommand{\dd}{\ensuremath{\mathrm{d}}}
\newcommand{\Eb}{E_{\bullet}}
\newcommand{\fb}{\ensuremath{\varphi_{\bullet}}}
\newcommand{\Tb}{\ensuremath{T_{\bullet}}}
\newcommand{\Folner}{F{\o}lner}
\newcommand{\llb}{\llbracket}
\newcommand{\rrb}{\rrbracket}
\newcommand{\llc}{\{\!\!\{}
\newcommand{\rrc}{\}\!\!\}}
\newcommand{\cc}[1]{\ensuremath{\llc{#1}\rrc}}
\newcommand{\lsub}[2]{\prescript{}{#1}{#2}}
\newcommand{\lsup}[2]{\prescript{[#1]}{}{#2}}
\newcommand{\Li}{\ensuremath{\mathscr{L}^{\infty}}}

\newcommand{\LOB}{\ensuremath{\Li_{\Omega,\BB}}}
\newcommand{\LO}{\ensuremath{\Li_{\Omega,\RR}}}

\newcommand{\LZ}{\ensuremath{\Li_{\ZZ,\RR}}}
\newcommand{\LZB}{\ensuremath{\Li_{\ZZ,\BB}}}
\newcommand{\LZH}{\ensuremath{\Li_{\ZZ,\cH}}}
\newcommand{\LZZ}{\ensuremath{\Li_{\ZZ^2,\RR}}}
\newcommand{\LZZH}{\ensuremath{\Li_{\ZZ^2,\cH}}}
\newcommand{\LZZB}{\ensuremath{\Li_{\ZZ^2,\BB}}}
\newcommand{\mub}{\ensuremath{\mu_{\bullet}}}
\newcommand{\MM}{\ensuremath{\mathfrak{M}}}
\newcommand{\MO}{\ensuremath{\mathfrak{M}_{\Omega}}}
\newcommand{\LZM}{\ensuremath{\Li_{\ZZ,\MM}}}
\newcommand{\muL}{\ensuremath{\mu_{\mathrm{L}}}}
\newcommand{\NN}{\mathbb{N}}
\newcommand{\Nrm}[1]{\left\|#1\right\|}
\newcommand{\nrm}[1]{\left|#1\right|}
\newcommand{\pair}[2]{\ensuremath{\langle{#1},{#2}\rangle}}
\newcommand{\Pair}[2]{\ensuremath{\llangle{#1},{#2}\rrangle}}
\newcommand{\PET}{\ensuremath{\mathbf{PET}}}
\newcommand{\aPET}{\ensuremath{\overline{\PET}}}
\newcommand{\Pfin}{\ensuremath{\mathcal{P}^*_{\!\mathrm{fin}}}}

\newcommand{\QQ}{\mathbb{Q}}
\newcommand{\RR}{\mathbb{R}}
\newcommand{\BS}{\mathbb{S}}
\newcommand{\ThL}{\ensuremath{{\Th_{\mathrm{Loeb}}}}}
\newcommand{\ThIntR}{\ensuremath{{\Th_{{\int:\RR}}}}}
\newcommand{\ThIntB}{\ensuremath{{\Th_{{\int:\BB}}}}}
\newcommand{\ThPET}{\ensuremath{\Th_{\PET}}}
\newcommand{\tL}{\ensuremath{\widetilde{\cL}}}
\newcommand{\tM}{\ensuremath{\widetilde{\cM}}}
\newcommand{\cT}{\ensuremath{\mathcal{T}}}
\newcommand{\UU}{\ensuremath{\mathrm{U}}}
\newcommand{\UH}{\ensuremath{\UU_{\cH}}}
\newcommand{\cW}{\ensuremath{\mathcal{W}}}
\newcommand{\ZwZ}{\ensuremath{\ZZ\wr\ZZ}}
\newcommand{\ZZ}{\ensuremath{\mathbb{Z}}}

\begin{document}
\title[Averages of unitary polynomial actions]
{Model theory and metric convergence~II:\\
  Averages of unitary polynomial actions}

\author
{Eduardo Dueñez \and José N. Iovino }

\address{Department of Mathematics\\
  The University of Texas at San Antonio\\
  One UTSA Circle\\
  San Antonio, TX 78249-0664\\
  U.S.A.}

\email{eduardo.duenez@utsa.edu}

\email{jose.iovino@utsa.edu}

\date{\today}
\thanks{We thank Xavier Caicedo, Christopher Eagle and Franklin Tall for their encouragement and feedback, as well as the Banff International Research Station for hosting the June 2016 FRG ``Topological Methods in Model Theory'' where many ideas in the Appendix to this manuscript were first conceived.}
\thanks{This research was funded by NSF grant DMS-1500615}

\subjclass[2010]{Primary: 37A30; Secondary: 03C98, 46Bxx, 28-xx}
\keywords{Mean Ergodic Theorem, PET induction, Leibman sequences, Henson structures}

\begin{abstract}
We use model theory of metric structures to prove the pointwise convergence, with a uniform metastability rate, of averages of a polynomial sequence $\{T_n\}$ (in Leibman's sense) of unitary transformations of a Hilbert space.
As a special case, this applies to unitary sequences $\{U^{p(n)}\}$ where $p$ is a polynomial $\ZZ\to\ZZ$ and $U$ a fixed unitary operator;
however, our convergence results hold for arbitrary Leibman sequences. 
As a case study, we show that the non-nilpotent ``lamplighter group''~$\ZZ\wr\ZZ$ is realized as the range of a suitable quadratic Leibman sequence.
We also indicate how these convergence results generalize to arbitrary \Folner\ averages of unitary polynomial actions of any abelian group~$\GG$ in place of~$\ZZ$. 
\end{abstract}

\maketitle

\section*{Introduction}
\label{sec:intro}

The first result on ``mean'' convergence of averages was von Neumann's 1932 Mean Ergodic Theorem~\cite{Neumann1932}:

\begin{MET}
  For any unitary operator $U$ on a Hilbert space~$\cH$ and any $x\in\cH$, the sequence $\AVb(x) = (\AV_n(x):n\in\NN)$ of pointwise averages
  \begin{equation*}
    \AV_n(x) = \frac{1}{n}\sum_{i=1}^n U^i(x)
  \end{equation*}
  converges as $n\to\infty$.
  The limit is equal to the orthogonal projection of~$x$ on the space of vectors fixed by~$U$.
\end{MET}

Historically, generalizations of von Neumann's theorem have largely followed a path influenced by a measure-theoretic viewpoint that is completely absent from the formulation above as a statement about convergence in Hilbert spaces.
We provide further historical background below.
Leaving history and measure theory aside for the moment, one may suggest the following different possible directions of generalization for MET:
\begin{enumerate}
\item Replace the sequence $(U^i:i\in\NN)$ with a ``higher-degree'' sequence $(U^{p(i)}:i\in\NN)$ where $p$ is a fixed polynomial.
\item The sequence $(T_i) = (U^{p(i)})$ above necessarily satisfies the commutativity condition $T_i\circ T_j = T_j\circ T_i$ for all $i,j$.
  To what extent can such commutativity requirement be removed?
\item What conditions on a family $(T_i)$ of unitary operators indexed by a semigroup other than~$\NN$ ensure the pointwise convergence of suitable averages?
\end{enumerate}
 Theorem~\ref{thm:PolyMET} in this manuscript is arguably the most natural generalization of von Neumann's result simultaneously in all three directions above.
(For technical reasons, Theorem~\ref{thm:PolyMET} is proved in the context of (polynomial) actions of \emph{groups} rather than semigroups.)
Theorem~\ref{thm:AbelPolyMET}, stated below, is a very particular case of more general results (Theorems~\ref{thm:PolyMET-Z}, \ref{thm:MetaPolyMET-Z} and~\ref{thm:PolyMET}). 
However, it is easiest to formulate and already generalizes MET all the way in direction~(1) and beyond.

\begin{theorem}[MET for abelian unitary polynomial actions of~$\ZZ$]\label{thm:AbelPolyMET}
  Fix $d\in\NN$.
  Let $\cH$ be a Hilbert space, and let $U_0,U_1,\dots,U_d$ be pairwise-commuting unitary operators on~$\cH$.
  For every $x\in\cH$, the sequence $\AVb(x) = (\AV_n(x) : n\in\NN)$ of averages%
\footnote{Here, ${k\choose j} = k(k-1)\cdots(k-j+1)/j!$ is the $j$-th binomial coefficient.}
  \begin{equation*}
    \AV_n(x) = \frac{1}{n+1}\sum_{0\le k\le n} U_0\circ U_1^k\circ U_2^{k\choose 2}\circ\dots\circ U_d^{k\choose d}(x)
  \end{equation*}
  converges as $n\to\infty$.

  In particular, if $p:\ZZ\to\ZZ$ is a polynomial of degree at most~$d$ and $U$ is a unitary operator on~$\cH$, then $\left( \sum_{0\le k\le n} U^{p(n)}(x)/(n+1) : n\in\NN \right)$ converges.

  Furthermore, there exists a universal metastability rate (depending only on~$d$) that applies uniformly to all sequences of averages of arbitrary~$x$ in the unit ball of an arbitrary Hilbert space~$\cH$ under arbitrary unitary operators $U_0,U_1,\dots,U_d$ on~$\cH$.
\end{theorem}

The notion of \emph{uniformly metastable convergence} above was first introduced in ergodic theory by Tao.
It is a main theme of our prior manuscript, but shall presently play a minor role~\cite{Duenez-Iovino:2017,Tao:2008,Tao2012}.

Taking a step in direction~(2), pairwise commutativity is not a necessary assumption; 
the sequence of averages under a family $(T_i)$ converges provided $i\mapsto T_i$ is a \emph{Leibman polynomial sequence} in the group $\UH$ of unitary operators on~$\cH$ (Theorems~\ref{thm:PolyMET-Z} and~\ref{thm:MetaPolyMET-Z}), but the range of this sequence need not generate an abelian group. 
The definition of Leibman polynomial sequence (Definition~\ref{def:Leibman-seq}) is motivated by the familiar fact that degree-$d$ polynomials $\RR\to\RR$ are characterized as those functions having $(d+1)$-iterated finite differences equal to zero.
The same essential definition gives the notion of \emph{Leibman polynomial mapping} from an arbitrary group~$\GG$ into~$\UH$~\cite{Leibman:2002}.
Theorem~\ref{thm:PolyMET} generalizes von Neumann's result in direction~(3) for Leibman polynomials $(T_i:i\in\GG)\subset\UH$ on abelian groups~$\GG$ endowed with a notion of averaging provided by a countable \Folner\ net. 

Continuing our historical remarks, the formulation of von Neumann's result above hides its conceptual genesis via the study of convergence of averages of square-integrable functions $f\in\sL^2(\Omega)$ on a probability space $(\Omega,\mu)$ under the action of a measure-preserving transformation $T$ of~$\Omega$.  
In this setting, MET asserts that the sequence $\AVb(f)$ of averages
\begin{equation*}
  \AV_n(f) = \frac{1}{n}\sum_{i=1}^n f\circ T^i
\end{equation*}
converges in~$\sL^2(\Omega)$ (after all, $f\mapsto f\circ T$ is a unitary transformation of~$\sL^2(\Omega)$).
This particular case of von Neumann's result explains why it is called a convergence result ``in mean'', i.e., in the mean-square (``$\sL^2$'') sense.
(By contrast, Birkhoff's Ergodic Theorem asserts the almost-everywhere pointwise convergence of the averages $\AV_n(f)$ for any $f\in\sL^1(\Omega)$~\cite{Birkhoff1931}.)
The $\sL^2$ setting entails no loss of generality since every Hilbert space~$\cH$ is realized as a space of square-integrable functions. 
However, this viewpoint is artificial for purposes of studying convergence under unitary actions (at least insofar as \emph{simple} actions are concerned, in contrast to multiple actions mentioned below).

Although generalizations of MET in direction~(1) seem very natural, we are not aware of direct proofs of Theorem~\ref{thm:AbelPolyMET}, but only of indirect proofs as byproduct of results on mean convergence of ``multiple'' ergodic averages.
Starting in the 1970's, Furstenberg pioneered the ergodic study of actions of \emph{multiple} simultaneous transformations; equivalently, the study of convergence of ``multiple averages'' of the \emph{product} of two or more measurable bounded functions on a probability space~$\Omega$ as acted upon by powers of measure-preserving transformations.
As an application of multiple averages, Furstenberg obtained a purely ergodic proof of Szemerédi's Theorem on the existence of arbitrary long arithmetic progressions in positive-density subsets of the integers~\cite{Furstenberg:1977,Szemeredi1975}.
However, Furstenberg's seminal results from the seventies did not extend von Neumann's theorem in either of the directions (1)--(3).
It was Bergelson who, in 1987, first extended some of Furstenberg's results to multiple ergodic averages of (\emph{plus quam} linear) polynomial powers of a fixed measure-preserving transformation acting on products of functions~\cite{Bergelson:1987}.
When specialized to simple measure-preserving actions, Bergelson's results are a step toward generalizing von Neumann's MET in direction~(1). 
However, there is no purely Hilbert-theoretical formulation of Bergelson's weak mixing hypothesis:
Even the convergence of pointwise averages of~$(U^{p(n)})$ stated in Theorem~\ref{thm:AbelPolyMET} only follows unconditionally from 2005 results for multiple ergodic averages of Host and Kra, and of Leibman (which depend on no mixing assumptions)~\cite{Host-Kra:2005,Leibman:2005}.

To our knowledge, Walsh's theorem~\cite{Walsh:2012} on mean convergence of nilpotent ergodic averages is the first result in the literature from which Theorem~\ref{thm:AbelPolyMET} follows as a corollary.
(Pointwise convergence of averages of~$(U^n\circ V^{n^2})$ under the assumption $U\circ V = V\circ U$ is a special case of 2009 results of Austin~\cite{Austin:2015a,Austin:2015b}.)
Thus, Walsh's theorem actually implies the convergence of averages asserted in the more general Theorem~\ref{thm:PolyMET-Z}, but only under the additional explicit hypothesis that $(T_i)$ generates a nilpotent subgroup of~$\UH$. 
However, our methods do not require a nilpotence hypothesis, but only the more intrinsic property that $(T_i)$ be a Leibman sequence in the sense of Definition~\ref{def:Leibman-seq-PET} (or in Leibman's more general sense of \emph{polynomial mapping} used in Theorem~\ref{thm:PolyMET}).  
In Section~\ref{sec:QuadLeib}, we construct a quadratic Leibman sequence whose range generates the non-nilpotent ``lamplighter group''~$\ZwZ$.

Generalizations of Walsh's theorem by Austin and Zorin-Kranich imply steps in direction~(3)~\cite{Austin:2016,ZorinKr:2016}. 
However, Theorems~\ref{thm:PolyMET-Z}, \ref{thm:MetaPolyMET-Z} and~\ref{thm:PolyMET} appear to be new in the general form stated. 
Nevertheless, given the close relation of our results to others in the existing literature, the main novelty is our ``soft'' direct approach to proving pointwise convergence of polynomial averages in Hilbert spaces using the framework of Henson metric structures.
Our viewpoint is heavily influenced by Tao's outline~\cite{Tao2012} of a nonstandard proof \emph{à la Robinson} of Walsh's theorem (although we use only standard real numbers, and none of Robinson's apparatus as such).
A significant part of the manuscript consists of natural definitions and basic results on model-theoretic notions of integration and convergence that parallel classical ones; 
nevertheless, we capture, refine, and in some cases extend such results in Henson's framework.
Section~\ref{sec:PET-struct} contains the rather long definition of the Henson class of \emph{PET structures over~\ZZ}.
Section~\ref{sec:Leibman-seqs} introduces the notion of \emph{Leibman polynomial sequence;}
it also exhibits a quadratic Leibman sequence whose range generates the non-nilpotent group~$\ZwZ$.
In Section~\ref{sec:poly-PET}, we state and prove Theorems~\ref{thm:PolyMET-Z} and~\ref{thm:MetaPolyMET-Z} on metastable convergence of polynomial unitary averages for Leibman sequences (over~\ZZ), and also explain how Theorem~\ref{thm:AbelPolyMET} follows as an immediate corollary.
In Section~\ref{sec:general}, we state and prove the most general of our ergodic convergence results in the form of Theorem~\ref{thm:PolyMET}, which generalizes MET in all three directions (1)--(3).
A number of foundational results are contained in the Appendix, which bears a close relation to our prior manuscript~\cite{Duenez-Iovino:2017}.
These results pertain to measure theory and integration of real functions, as well as abstract notions of integration of functions taking values in Banach spaces.
In this way we obtain a Dominated Convergence Theorem for notions of integration in an \emph{ad hoc} Henson class of Banach integration frameworks (Theorem~\ref{thm:DCT}).
We also show that the compactness of Henson's logic implies a Uniform Metastability Principle for convergence in models of any Henson theory (Proposition~\ref{thm:UMP}).
Via this principle, all our results on convergence of averages admit refinements to convergence with metastability rates that are universal.
These are \emph{gratis} refinements thanks to the model-theoretic approach.

\section{PET Structures}
\label{sec:PET-struct}

\subsection{Classical PET Structures}
\label{sec:PET-classical}

\begin{notation}\label{def:notation-sorts}
Below we list a number of \emph{formal symbols} $\RR, \ZZ, \NN, \cH, \dots$ that will eventually become \emph{sort descriptors} for a Henson language of metric structures.  
However, throughout this subsection, these symbols have the following classical interpretations:
  \begin{itemize}
  \item $\RR,\ZZ,\NN$ shall denote the sets of real numbers, integers and naturals.
  \item $\cH$ shall denote a real Hilbert space.
  \item $\BB$ shall denote the real Banach algebra $\BB(\cH,\cH)$ of bounded operators on~$\cH$.
  \item $\AZ$ shall denote the Boolean algebra of all subsets of~$\ZZ$.
  \item $\MM$ shall denote the real Banach space of signed finite measures on~$\ZZ$ (i.e., on the measure space $(\ZZ,\AZ)$).
  \item $\LZ$ shall denote the Banach space $\Li(\ZZ,\RR)$ of bounded real functions on~$\ZZ$.
  \item $\LZH$ shall denote the Banach space $\Li(\ZZ,\cH)$ of bounded functions $\ZZ\to\cH$.
  \item $\LZB$ shall denote the Banach space $\Li(\ZZ,\BB)$ of bounded functions $\ZZ\to\BB$.
  \item $\LZM$ shall denote the Banach space $\Li(\ZZ,\MM)$ of bounded functions $\ZZ\to\MM$.
  \item $\LZZ$ shall denote the Banach space $\Li(\ZZ\times\ZZ,\RR)$ of bounded real functions on~$\ZZ\times\ZZ$.
  \item $\LZZH$ shall denote the Banach space $\Li(\ZZ\times\ZZ,\cH)$ of bounded functions~$\ZZ\times\ZZ\to\cH$.
  \item $\LZZB$ shall denote the Banach space $\Li(\ZZ\times\ZZ,\BB)$ of bounded functions~$\ZZ\times\ZZ\to\BB$.
  \end{itemize}
\end{notation}

From a model-theoretic viewpoint, the sets $\RR,\NN,\ZZ,\dots$ denoted by the formal symbols above are the \emph{sorts} of a metric Henson structure~$\cM$. 
(Discrete sorts $\NN$, $\ZZ$, $\AZ$ are still viewed as metric spaces endowed with the discrete metric.)
In addition, $\cM$ is endowed with a number of distinguished elements (``constants'') and continuous functions between sorts.  
The distinguished elements include:
\begin{itemize}
\item All elements of~$\NN$ and~$\ZZ$.
\item All rational numbers in~$\RR$.
\item The zero element of each real Banach space above ($\cH, \BB, \MM, \LZ, \dots$).
\item The identity operator $I\in\BB$.
\item The zero (empty set $\emptyset$) and unity (improper subset $\ZZ\subseteq\ZZ$) of the Boolean algebra~$\AZ$.
\end{itemize}
The distinguished functions between sorts include:
\begin{itemize}
\item The discrete metric in each the discrete sorts~$\ZZ$, $\NN$, $\AZ$.
  \item The operations of addition, subtraction, multiplication, absolute value, and lattice operations (binary minimum and maximum) on~$\RR$.
\item The order $\le$ of~$\NN$, identified with its characteristic function $\llb\cdot\!\le\!\cdot\rrb : \NN\times\NN\to\{0,1\}$.
\item The membership relation from $\ZZ$ to $\AZ$, identified with its characteristic function $\bin{\cdot}{\cdot} : \ZZ\times\AZ\to\{0,1\}$.
  \item The group operations (unary negation, binary addition and subtraction) of~$\ZZ$.
\item The operations of union, intersection and complementation on~$\AZ$.
\item The Hilbert space operations (addition, scalar multiplication, and inner product~$(x,y)\mapsto x\cdot y$) on~$\cH$.  
For convenience, also the norm $\Nrm{x} = \sqrt{x\cdot x}$.
  \item The operations of addition and scalar product, and the Banach norm~$\Nrm{\cdot}$ on each Banach sort $\BB,\MM,\Li_{X,Y}$.
\end{itemize}
For $f\in\Li_{X,Y}$, the Banach norm is $\Nrm{f} = \sup_{x\in X}\Nrm{f(x)}$, where $\Nrm{f(x)}$ is the norm of $f(x)$ as an element of Banach sort~$Y$.
The Banach norm on~$\BB$ is $\Nrm{T} = \sup\{\Nrm{T(x)} : x\in\cH, \Nrm{x}\le 1\}$.
The Banach norm on~$\mu\in\MM$ is ``total variation'': 
Recall that $\mu$ has an atomic decomposition $\mu = \sum_{i\in\ZZ}c_i\delta_i$ where $\delta_i$ is the unit mass at~$i$ and $c_i = \mu(\{i\})$.
With this notation, $\Nrm{\mu} = \sum_i|c_i|$.

(To abbreviate the long list of distinguished functions, above and in what follows we use $X$ to denote either of the ``domain'' discrete sets $\ZZ$, $\ZZ^2$ of the various sorts $\Li$, and $Y$ to denote the ``codomain'' Banach sorts $\RR$, $\cH$, $\BB$, $\MM$.)

The list of distinguished functions continues as follows:

\begin{itemize}
\item The operations $\Li_{X,\cH}\times\Li_{X,\cH}\to\Li_{X,\RR}$ induced by (pointwise) application of the inner product of~$\cH$.
\item The unary operation of pointwise absolute value $\nrm{\cdot}$ and the binary lattice operations (pointwise $\max$ and $\min$) on sorts $\Li_{X,\RR}$.
\item The unary operation~$\nrm{\cdot}$ of \emph{measure of total variation} and the binary lattice operations (``pointwise'' $\max$ and $\min$) on~$\MM$ (i.e., $\nrm{\mu} = \sum_i \nrm{a_i}\delta_i$, $\max(\mu,\nu) = \sum_i\max\{a_i,b_i\}\delta_i$, and $\min(\mu,\nu) = \sum_i\min\{a_i,b_i\}\delta_i$ if $\mu = \sum_i a_i\delta_i$ and $\nu = \sum_i b_i\delta_i$).
\item The operation of pointwise magnitude $\nrm{\cdot} : \Li_{X,Y}\to\Li_{X,\RR}$, namely $|f| : x\mapsto \Nrm{f(x)}$ for any $f\in\Li_{X,Y}$.
  \item The unary adjoint operation $T\mapsto T^{*}$ on~$\BB$, and the corresponding induced operations (pointwise adjoint) on sorts~$\Li_{X,\BB}$.
  \item The binary operation $(S,T)\mapsto S\circ T$ of composition on~$\BB$, and the corresponding induced operations of pointwise composition on sorts~$\Li_{X,\BB}$.
  \item The inclusions:
    \begin{itemize}
    \item $\ZZ\hookrightarrow\AZ : i\mapsto\{i\}$.
    \item $\AZ\hookrightarrow\LZ : A\mapsto\chi_A$ where $\chi_A$ is the characteristic function of the subset~$A\subseteq\ZZ$.
  \item $\ZZ\hookrightarrow\MM$ given by $i\mapsto\delta_i$ (the unit point mass at~$i$).
    \item $Y\hookrightarrow\Li_{X,Y}$, with $y\in Y$ identified with the constant function $y(\blacksquare):x\mapsto y$ in~$\Li_{X,Y}$;
    \item The right inclusion map $\Li_{\ZZ,Y}\hookrightarrow\Li_{\ZZ^2,Y}$ whereby $f\in\Li_{\ZZ,Y}$ is identified with $f(\blacksquare,\cdot) : (w,x)\mapsto f(x)$; also, the analogous left inclusion map identifying~$f$ with $f(\cdot,\blacksquare) : (w,x)\mapsto f(w)$.
    \end{itemize}
  \item The function-evaluation maps
    \begin{itemize}
    \item $(T,x)\mapsto T(x)$ from $\BB\times\cH$ to~$\cH$.
    \item $(f,x)\mapsto f(x)$ from $\Li_{X,Y}\times X$ to~$Y$;
    \end{itemize}
   Also, the maps $\Li_{X,\BB}\times\Li_{X,\cH}\to\Li_{X,\cH}$ induced by pointwise evaluation.
  \item The partial evaluation maps:
    \begin{itemize}
    \item Left evaluation $\Li_{\ZZ^2,Y}\times\ZZ\to\Li_{\ZZ,Y}$, namely $(F,i)\mapsto F(i,\cdot)$ where $F(i,\cdot) : j\mapsto F(i,j)$.
    \item Right evaluation $\Li_{\ZZ^2,Y}\times\ZZ\to\Li_{\ZZ,Y}$, namely $(F,j)\mapsto F(\cdot,j)$ where $F(\cdot,j) : i\mapsto F(i,j)$.
    \end{itemize}
  \end{itemize}
(Note that the left evaluation map allows us to identify $\Li_{\ZZ^2,Y}$ with the space $\Li(\ZZ,\Li_{\ZZ,Y})$ of all bounded functions $\ZZ\to\Li_{\ZZ,Y}$---thus making a potential sort $\Li(\ZZ,\Li_{\ZZ,Y})$ superfluous. 
We also have a different identification of $\Li(\ZZ,\Li_{\ZZ,Y})$ with $\Li(\ZZ^2,Y)$ via right evaluation.) 
  \begin{itemize}
  \item The \emph{\Folner-measure map} $\sigma : \NN\to\MM$, where 
    \begin{equation*}
      \sigma_n = \frac{1}{n+1} \sum_{0\le i\le n}\delta_i\qquad
      \text{for all $n\in\NN$.}
    \end{equation*}
    ($\sigma_n$ is the average of unit point masses at the points $0,1,2,\dots,n$.)
  \item The translation action of~$\ZZ$ on~$\Li_{\ZZ,Y}$. 
We regard this action as a function $\Li_{\ZZ,Y}\to \Li_{\ZZ^2,Y} = \Li(\ZZ,\Li_{\ZZ,Y})$ (with the latter identification by partial evaluation on the left). 
The action is denoted $f\mapsto\lsub{\bullet}f$ where $\lsub{\bullet}f \in \Li(\ZZ,\Li_{\ZZ,Y})$ is the function $i\mapsto\lsub{i}f$ with $\lsub{i}f\in\Li_{\ZZ,Y}$ the function $j\mapsto f(i+j)$.
  \item The \emph{shear transformation} $\Li_{\ZZ^2,Y}\to\Li_{\ZZ^2,Y}$, namely $F\mapsto \widetilde{F}$ where $\widetilde{F}:(i,j) \mapsto F(i,i+j)$.
  \item The translation action of~$\ZZ$ on~$\MM$, regarded as a mapping $\MM\to\LZM$ and denoted $\mu\mapsto\lsub{\bullet}\mu$ where $\lsub{\bullet}\mu\in\LZM$ is the mapping $i\mapsto\lsub{i}\mu$, with $\lsub{i}\mu\in\MM$ the measure $\mu$ shifted by~$-i$, namely
    \begin{equation*}
      \lsub{i}\mu = \sum_{j\in\ZZ} a_j\delta_{j+i}\qquad
\text{if}\ \mu = \sum_{j\in\ZZ} a_j\delta_j,
    \end{equation*}
    which is classically characterized by the property that $\pair{f}{\mu} = \pair{\lsub{i}f}{\lsub{i}\mu}$ for all $f\in\LZ$ and $i\in\ZZ$.
  \item The involutions $\Li_{\ZZ^2,Y}\to\Li_{\ZZ^2,Y}$ induced by the involution $(i,j)\mapsto(j,i)$ of $\ZZ^2$.
  \item The integration operations 
    \begin{itemize}
    \item $\Li_{\ZZ,Y}\times\MM \to Y : (f,\mu)\mapsto\pair{f}{\mu} = \sum_{i\in\ZZ}c_if(i)$ for $\mu = \sum_ic_i\delta_i$.
    \item (Left integral) $\MM\times\Li_{\ZZ^2,Y} \to \Li_{\ZZ,Y} : (\mu,F)\mapsto\Pair{\mu}{F}$, where $\Pair{\mu}{F}\in\Li_{\ZZ,Y}$ is the function $j\mapsto\pair{F(\cdot,j)}{\mu} = \sum_ic_iF(i,j)$.
    \item (Right integral) $\Li_{\ZZ^2,Y}\times\MM \to \Li_{\ZZ,Y} : (F,\mu)\mapsto\Pair{F}{\mu}$, where $\Pair{F}{\mu}\in\Li_{\ZZ,Y}$ is the function $i\mapsto\pair{F(i,\cdot)}{\mu} = \sum_jc_jF(i,j)$.
    \item $\Li_{\ZZ^2,Y}\times\LZM \to \Li_{\ZZ,Y} : (F,\mub)\mapsto\Pair{F}{\mub}$ where $\Pair{F}{\mub}\in\Li_{\ZZ,Y}$ is the function $j \mapsto \pair{F(\cdot,j)}{\mu_j}$ (i.e., the operation induced by ``pointwise integration'' when $\Li_{\ZZ^2,Y}$ is identified with $\Li(\ZZ,\Li_{\ZZ,Y})$ via left partial evaluation).
    \end{itemize}
  \end{itemize}
For visual convenience, we may use integral notation and write $\int\! f\,d\mu$ or $\int\!f(i)\,d\mu(i)$ for $\pair{f}{\mu}$, and $\int\!F(i,\cdot)\,d\mu(i)$ for $\Pair{\mu}{F}$ (resp., $\int\!F(\cdot,j)\,d\mu(j)$ for $\Pair{F}{\mu}$).
  \begin{remarks}
    \begin{itemize}
    \item There are redundancies on the list of functions above.
    For instance, the $\sL^2$-norm on~$\cH$ is implicitly defined by its inner product: $\Nrm{x}^2 = x\cdot x$.
    As a less trivial example, the action of $\ZZ$ on~$\Li_{\ZZ,Y}$ is obtained from the right inclusion $\Li_{\ZZ,Y}\hookrightarrow\Li_{\ZZ^2,Y}$ followed by the shear transformation.
    However, for reasons of exposition we make no effort to present a minimal list of distinguished functions.
The model-theoretic approach fundamentally requires that all sorts, functions and constants that are relevant to the problem at hand be part of the structures under study.

  \item The nonstrict order relations ($\le$ and $\ge$) of~$\RR$ are the only predicate symbols of a Henson language.  
However, any discrete predicate $P$ may be identified with a $\{0,1\}$-valued function $\chi_P$ (the characteristic function of the truth set of~$P$), so the usual interpretation of~$P(x)$ (resp., of~$\neg P(x)$) agrees with the interpretation of the Henson formula $\chi_P(x)\ge 1/2$ (resp., of~$\chi_P(x)\le 1/2$).
\end{itemize}
\end{remarks}

\begin{definition}[Classical PET structure over~\ZZ]
\label{def:PET-classical}
A \emph{classical PET structure (over~\ZZ)} is a triple $\cM = (\mathbf{S},\mathbf{C},\mathbf{F})$ where
\begin{equation*}
\mathbf{S} = (\RR,\NN,\ZZ,\AZ,\cH,\BB,\MM,\LZ,\LZH,\LZB,\LZM,\LZZ,\LZZH,\LZZB)
\end{equation*}
is a collection of \emph{sorts}, $\mathbf{C}$ is a collection of distinguished elements (\emph{constants}), and $\mathbf{F}$ is a collection of distinguished functions between sorts, provided these sorts, constants and functions are obtained in the manner prescribed by Notation~\ref{def:notation-sorts}.
\end{definition}

\subsection{Abstract PET structures}
\label{sec:PET-abstract}

\begin{definition}[Henson signature and language for PET structures over~$\ZZ$]
\label{def:PET-signature}
The \emph{Henson signature for PET structures over~$\ZZ$} consists of three ingredients:
\begin{itemize}
\item A collection of formal symbols, called \emph{sort descriptors} (or \emph{sort names}) in one-to-one correspondence with the collection~$\mathbf{S}$ of sorts of a classical PET structure. 
For definiteness, the collection of descriptors is taken to be
\begin{equation*}
  (\RR,\NN,\ZZ,\AZ,\cH,\BB,\MM,\LZ,\LZH,\LZB,\LZM,\LZZ,\LZZH,\LZZB)
\end{equation*}
its members regarded as purely formal symbols.
\item A collection of lexical \emph{constant symbols} containing a unique symbol $\mathtt{c}$ for each of the distinguished elements in Definition~\ref{def:PET-classical}, with each such symbol endowed with a sort descriptor~$s$ naming that sort to which the element $c$ named by~$\mathtt{c}$ belongs per Definition~\ref{def:PET-classical}.
\item A collection of lexical \emph{function symbols} containing a unique symbol $\mathtt{f}$ for each of the functions named in Definition~\ref{def:PET-classical}, with each such symbol endowed with a \emph{sort-specification} of the form $s_1\times\dots\times s_n\to s_0$ where $s_0,s_1,\dots,s_n$ are sort descriptors chosen in accordance with the specification of the domain (Cartesian product of sorts named by $s_1,\dots,s_n$) and codomain (sort named by~$s_0$) of the function~$f$ named by the symbol~$\mathtt{f}$.
\end{itemize}
The \emph{Henson language $\cL$ for PET structures over~\ZZ} is the Henson language (of positive bounded formulas) whose signature is the one just described~\cite{Henson-Iovino:2002,Iovino:2014,Duenez-Iovino:2017}.
\end{definition}

\begin{definition}[PET structure over~$\ZZ$]
\label{def:PET-abstract}
Let $\cL$ be the Henson language for PET structures. 
Let $\PET$ be the class of all classical PET structures over~$\ZZ$ per Definition~\ref{def:PET-classical}, and let ${\ThPET}$ be the $\cL$-theory of \PET\ in Henson's logic of approximate satisfaction of positive bounded formulas.
An \emph{(abstract) PET structure over~$\ZZ$} is a model of~${\ThPET}$.  
\end{definition}
The class \aPET\ of abstract PET structures obviously extends~\PET.

\begin{remarks}
  \begin{itemize}
  \item In principle, one may provide an explicit axiomatization in positive bounded Henson formulas of the class~\PET.
However, given the large number of sorts and functions in a PET structure this task is impractical. 
We refer the reader to our prior manuscript in which we provide explicit Henson axiomatizations of certain classes of structures somewhat more general than \PET~\cite{Duenez-Iovino:2017}.
Nevertheless, it should be clear that the Henson theory ${\ThPET}$ is \emph{uniform} in the sense that it imposes bounds on constants as well as local bounds and local moduli of uniform continuity on distinguished functions. 
Moreover, ${\ThPET}$ obviously is identical to the theory $\Th_{\aPET}$ of all abstract PET structures. 

\item The \Folner\ map $\sigma : \NN\to\MM$ per Notation~\ref{def:notation-sorts} implies a particular choice of a ``notion of averaging'' over~$\ZZ$ that is built into~${\ThPET}$. 
Nonequivalent definitions of the PET class over~$\ZZ$ and of $\aPET$ are obtained by changing this choice (e.g., letting $\sigma_n = 1/(2n+1)\sum_{-n\le i\le n}\delta_i$ in classical structures), but Theorems~\ref{thm:PolyMET-Z} and~\ref{thm:MetaPolyMET-Z} on PET structures over~$\ZZ$ remain true under such alternate choice (in fact, they are special cases of the more general Theorem~\ref{thm:PolyMET}).

\item If $\cM$ is a PET structure, then the $\RR$-named sort $\RR^{\cM}$ of~$\cM$,  under the corresponding operations $+_{\RR}^{\cM}, -_{\RR}^{\cM}, \dots$, is (isomorphic to) the standard real numbers;
  we shall identify~$\RR^{\cM}$ with~$\RR$.
Correspondingly, the ``Hilbert sort''  $\cH^{\cM}$ of~$\cM$ is a classical real Hilbert space.
Typically, the $\NN$-named sort $\cN = \NN^{\cM}$ of $\cM$ is a proper extension of the set $\NN$ of standard natural numbers (when the latter is identified with the set of interpretations~$\mathtt{m}^{\cM}$ of the constant symbols~$\mathtt{m}$ of~$\cL$, one for each standard natural~$m$)%
\footnote{The language $\cL$ has constants naming only the standard integers and natural numbers, but no nonstandard elements of the sorts $\NN^{\cM}$, $\ZZ^{\cM}$.}, and similarly~$\cZ=\ZZ^{\cM}$ extends~$\ZZ$ in general.
While $\BB^{\cM}$ may be identified (via the evaluation map $\BB\times\cH\to\cH$) with an algebra of bounded operators on~$\cH^{\cM}$, it need not contain all bounded operators.
The sort $\AZ^{\cM}$ may be identified (via $\llb\cdot\in\cdot\rrb$) with a Boolean algebra of some, but not necessarily all subsets of~$\ZZ^{\cM}$, while $(\LZ)^{\cM}$ may be identified (via the evaluation map) with a space of (not necessarily all) bounded functions $\cZ\to\RR$.

One of the subtlest differences between classical and abstract PET structures is the fact that $\MM^{\cM}$ typically consists of measures that are finitely but not countably additive on~$\ZZ^{\cM}$ (in particular, such measures need not have atomic decompositions as in the classical case).  
Fortunately, this difference turns out not to be critical, at least if one works in saturated PET structures:  
In this setting, the interplay between sorts $\ZZ^{\cM}$, $\AZ^{\cM}$ and~$(\LZ)^{\cM}$ comes to the rescue via analogues of Loeb measure and Loeb integration~\cite{Duenez-Iovino:2017}.
Appendix~\ref{sec:Loeb} contains a basic discussion of Loeb structures.
\end{itemize}
\end{remarks}

\section{ Leibman sequences}
\label{sec:Leibman-seqs}

\subsection{Classical Leibman sequences}
\label{sec:Leibman-classical}

Leibman introduced the notion of \emph{polynomial sequences} in a group~$G$~\cite{Leibman:1998}.
Leibman's polynomial sequences in~$G$ generalize sequences (indexed by~$\ZZ$) of the form $(g^{p(j)}:j\in\ZZ)$ where $p : \ZZ\to\ZZ$ is a polynomial and $g\in G$ is fixed.
Fix such a sequence ${\Tb} = (T_j:j\in\ZZ)$ where $T_j = g^{p(j)}$.
For fixed~$i\in\ZZ$, the sequence $\Delta^{\!i}{\Tb} = (T_{i+j}\circ T_j^* : j\in\ZZ)$ of ``step-$i$ discrete differences'' of~${\Tb}$ is of the form $(g^{q(j)})$ where $q = \Delta^{\!i}p : j\mapsto p(i+j)-p(j)$ is a polynomial of degree less than~$p$ (or possibly the zero polynomial).  
This motivates Leibman's recursive definition of polynomial sequence as follows.

\begin{definition}[Discrete difference and Leibman sequence] \label{def:Leibman-seq}
  Let $G$ be a multiplicative group with identity~$I$.
  When convenient, the inverse $g^{-1}$ of an element $g$ of~$G$ will be denoted~$g^{*}$.
  Let $\GZ$ be the group of all $\ZZ$-sequences $T : j\mapsto T_j$ from $\ZZ$ into~$G$ under the operation of pointwise multiplication induced from~$G$, and endow~$\GZ$ with the translation action $\lsub{i}T$ of~$\ZZ$, namely $(\lsub{i}T)_j = T_{i+j}$.
  For $i\in\ZZ$, the \emph{discrete-difference operator} is the function $T\mapsto\Delta^{\!i}T := \lsub{i}T\cdot T^{*}$ from $\GZ$ to $\GZ$;
  it is uniquely characterized by the identity
\begin{equation*}
  (\Delta^{\!i}T)_j = T_{i+j}\cdot T^*_j \qquad\text{for all $j\in\ZZ$.}
\end{equation*}
(Since $(T^{*})_j = (T_j)^{*}$, parentheses may be omitted without ambiguity.)
We will also omit parentheses when writing iterated discrete differences;
thus, $\Delta^{\!i}\Delta^{\!j}T$ means $\Delta^{\!i}(\Delta^{\!j}T)$.

Let $\Bone$ denote the constant sequence $j\mapsto I$. 
Given $d\in\NN$, a \emph{Leibman sequence in~$G$ of degree at most~$d$} is any $T\in \GZ$ all of whose $(d+1)$-fold iterated discrete differences are trivial, i.e., 
\begin{equation*}
  \Delta^{i_d}\dots\Delta^{i_1}\Delta^{i_0}T = \Bone\qquad
\text{for all $i_0,i_1,\dots,i_d\in\ZZ$.}
\end{equation*}
A Leibman sequence is a Leibman sequence of any degree~$d$; its degree $\degL\! T$ is the least such~$d$. 
(We define formally $\degL\!\Bone = -\infty$.)
A Leibman sequence $T$ of degree at most~$0$ is called \emph{translation-invariant} or \emph{constant};
it is of the form $T = g\Bone$ for some $g\in G$ (i.e., $T_k = g$ for all $k\in\ZZ$).
\end{definition}

\begin{remarks}
  \begin{itemize}
  \item The definition of Leibman sequence above is indirect and recursive;
it involves only the group structures of~$(\ZZ,+)$ and~$(G,\cdot)$, but not the product of~$\ZZ$ as one might otherwise expect from the usual construction of polynomials starting with monomials built from multiplication.
\item It can be shown (by an application of the usual method of finite differences) that if $T$ is a Leibman sequence of degree at most~$d$ in an \emph{abelian} group~$G$, then there exist $g_0,g_1,\dots,g_d\in G$ such that
\begin{equation*}
  T_k = g_0\cdot g_1^k\cdot g_2^{k\choose 2}\cdot\ldots\cdot g_d^{k\choose d}\qquad
\text{for all $k\in\ZZ$,}
\end{equation*}
where ${k\choose j} = k(k-1)\cdots(k-j+1)/j!$ is the $j$-th binomial coefficient.%
\footnote{See Proposition~\ref{prop:quad-Leib-fmla} below for the case of sequences in~$G$ abelian that are at most quadratic.}
One may regard $g_0,g_1,\dots,g_d$ as the ``coefficients'' of the Leibman polynomial~$T$.
In particular, this abelian setting comprises all families $(g^{p(k)})$ where $p$ is a polynomial $\ZZ\to\ZZ$ and $g\in G$ is fixed.
Theorem~\ref{thm:AbelPolyMET} states the convergence of ergodic averages in the abelian case;
nevertheless, Theorems~\ref{thm:PolyMET-Z}, \ref{thm:MetaPolyMET-Z} and~\ref{thm:PolyMET} only assume that $T$ is a unitary Leibman sequence per Definition~\ref{def:Leibman-seq-PET}, but no additional explicit commutativity hypotheses. 
   \item Translations commute with inversion and with discrete differences, i.e., $(\lsub{j}T)^{*} = \lsub{j}(T^{*})$ and $\lsub{j}(\Delta^{\!i}T) = \Delta^{\!i}(\lsub{j}T)$.  
(The latter equality depends on the commutativity of addition on~$\ZZ$.)
In particular, Leibman degree is invariant under translation.
However, the discrete difference operators do \emph{not} commute with adjoints, so Leibman degree is not invariant under taking adjoints.
Correspondingly, $T^{*}$ need not be a Leibman polynomial if~$T$ is. 
\end{itemize}
\end{remarks}

Leibman sequences of degree at most~$1$ are easily characterized:

\begin{proposition}\label{prop:Leibman-linear}

  Given any fixed choice of $a,b\in G$, there exists a unique Leibman sequence~$T$ of degree at most~$1$ satisfying $b=T_0$ and $a = \Delta^{\!1}T_0$, namely $T : k \mapsto a^kb$.
\end{proposition}
The straightforward proof of Proposition~\ref{prop:Leibman-linear} is left to the reader.

\subsection{Quadratic Leibman sequences}
\label{sec:QuadLeib}

In this section we characterize classical Leibman sequences that are quadratic, i.e., of degree at most~$2$.
In particular, we construct a quadratic Leibman sequence whose range generates the non-nilpotent ``lamplighter'' group~$\ZwZ$ (Corollary~\ref{cor:quadr-Leibman-nonnil}).
Throughout this section, $G$ will denote a multiplicative group with identity~$I$;
the inverse~$g^{-1}$ of~$g\in G$ is denoted~$g^{*}$ when convenient.

In what follows, we fix a quadratic Leibman sequence~$T$.
One may suspect that $T$ is uniquely characterized by three constants, say $a = \Delta^{\!1}\Delta^{\!1}T_0$, $b = \Delta^{\!1}T_0$ and $c = T_0$;
this is easily shown to be true (See Proposition~\ref{prop:Leibman-quadratic} below).
However, in contrast to Proposition~\ref{prop:Leibman-linear}, the constants $a,b,c$ are not arbitrary:
The requirement that they correspond to a \emph{bona fide} Leibman sequence~$T$ imposes nontrivial relations among $a$ and~$b$:
They must generate a factor of~$\ZwZ$ by Propositions~\ref{prop:Leibman-quadratic} and~\ref{prop:quad-Leib-wreath}.

  \begin{proposition}\label{prop:Leibman-quadratic}
    Given a quadratic Leibman sequence~$T$ in a group~$G$, the elements $a = \Delta^{\!1}\Delta^{\!1}T_0 = T_2T_1^{*}T_0T_1^{*}$ and $b = \Delta^{\!1}T_0 = T_1T_0^{*}$ satisfy the commutation relations
    \begin{equation}
      \label{eq:commut-relns-G}
      a\cdot \lsup{b^k}\!a = \lsup{b^k}\!a\cdot a\qquad\text{for all $k\in\ZZ$,}
    \end{equation}
    where $\lsup{h}\!g := hgh^{*}$ is the conjugate of $g$ by~$h$.
    Conversely, given $a,b,c\in G$ such that the above relations hold for~$a$ and~$b$, there exists a unique quadratic Leibman sequence~$T$ satisfying $\Delta^{\!1}\Delta^{\!1}T_0 = a$, $\Delta^{\!1}T_0 = b$ and $T_0 = c$.
  \end{proposition}

  Note that the commutation relations~\eqref{eq:commut-relns-G} do not involve~$c$ at all.

\subsubsection{Proof of Proposition~\ref{prop:Leibman-quadratic}}
\label{sec:proof-Leibman-quadratic}

Free variables $i,j,k,l,m,n,p,x,y,z$ will denote elements of~$\ZZ$ throughout.
The iterated discrete differences $\Delta^{\!i}T, \Delta^{\!i}\Delta^{\!j}T, \Delta^{\!i}\Delta^{\!j}\Delta^{\!k}T$ of~$T$ will be denoted $\delta(i),\delta(i,j),\delta(i,j,k)$, respectively.
  
  First, let $T$ be a quadratic Leibman sequence;
  thus, $\delta(i,j,k) = \Bone$ for all $i,j,k$, by definition of Leibman degree.
  It follows that $\lsub{i}\delta(j,k) = \delta(j,k)$, i.e., $\delta(j,k)$ is constant.

  Denote by $\lsup{S}\!R$ the conjugate $SRS^{*}$ of $R$ by~$S$ ($R,S\in\GZ$).
  Straightforward algebra shows that the ``cocycle identity''
  \begin{equation}\label{eq:cocycle-gen}
    \delta(j,k+l) = \lsub{l}\delta(j,k)\cdot\lsup{\lsub{l}\delta(k)}\delta(j,l)
  \end{equation}
  holds for arbitrary $T\in\GZ$.
  Under the assumption that~$T$ is quadratic, all terms $\delta(\cdot,\cdot)$ in the identity above are constant, so the cocycle identity improves itself to one with an extra free parameter~$m$:
  \begin{equation}
    \label{eq:cocycle-quad}
    \delta(j,k+l) = \delta(j,k)\cdot\lsup{\lsub{m}\delta(k)}\delta(j,l).
  \end{equation}
  Let $\llb R,S\rrb = R^{*}S^{*}RS$ be the commutator of~$R,S$.
  Using the cocycle identity~\eqref{eq:cocycle-quad} to expand $\delta(i,j+k+l)$ in two different ways, we find $\llb \delta(i,j), \lsub{m}\delta(k+l)^{*}\cdot\lsub{n}\delta(k)\cdot\lsub{p}\delta(l)\rrb = \Bone$.
  Using the relation $\delta(x)^{*} = \lsub{x}\delta(-x)$, this identity may be rewritten
  \begin{equation}\label{eq:delta2-commutator}
  \llb
  \delta(i,j), \lsub{k}\delta(x)\cdot\lsub{l}\delta(y)\cdot\lsub{m}\delta(z)
  \rrb
  = \Bone\qquad\text{whenever $x+y+z=0$.}
\end{equation}
Let $H$ be the subgroup of~$\GZ$ generated by all elements $\lsub{k}\delta(x)$, and $K$ the subgroup of~$H$ generated by all elements~$\lsub{k}\delta(x)\cdot\lsub{l}\delta(y)\cdot\lsub{m}\delta(z)$ with $x+y+z=0$.
It follows from the commutation relations~\eqref{eq:delta2-commutator} that $K$ is a subgroup of the centralizer of~$\delta(i,j)$ in~$\GZ$.
It is easy to see that $K$ is a normal subgroup of~$H$, and  $x\mapsto\delta(x)\pmod{K}$ is a homomorphism $\ZZ\to H/K$.
Note that $\lsup{\delta(l)^k}\delta(m,n) = \delta(l)^k\cdot\lsub{m}\delta(n)\cdot\delta(n)^{-1}\cdot\delta(l)^{-k} \equiv \delta(kl)\cdot\delta(n)\cdot\delta(-n)\cdot\delta(-kl) \equiv \Bone\pmod{K}$, so the following commutation identity follows:
\begin{equation}\label{eq:commut-relns-GZ}
\delta(i,j) \cdot \lsup{\delta(l)^k}\delta(m,n) = 
\lsup{\delta(l)^k}\delta(m,n) \cdot  \delta(i,j).
\end{equation}
Putting $i=j=l=m=n=1$ (with $k$ arbitrary) and evaluating at~$0$, we obtain the commutation relations~\eqref{eq:commut-relns-G}.

Conversely, let $a,b,c$ satisfy the commutation relations~\eqref{eq:commut-relns-G}.
We claim that a unique $T\in\GZ$ exists satisfying 
\begin{align}  \label{eq:delta11-T}
  T_0 &= c, & \delta(1)_0 &= b, & 
        \delta(1,1) &= a\Bone\qquad\text{(i.e., $\delta(1,1)_k = a$ for all $k\in\ZZ$)}.
\end{align}
Let $T_0=c$ and $T_1=bc$, so the first two conditions above hold.
For $k\ge 0$, the condition $T_{k+2}T_{k+1}^{*}T_kT_{k+1}^{*} = \delta(1,1)_k=a$ is equivalent to the forward recurrence $T_{k+2} = aT_{k+1}T_k^{*}T_{k+1}$, while for $k\le 0$, it is equivalent to the backward recurrence $T_k = T_{k+1}T_{k+2}^{*}aT_{k+1}$.
Using both recurrences with the initial values $T_0=c$, $T_1=bc$, we obtain a unique $T\in\GZ$ satisfying the conditions~\eqref{eq:delta11-T}.

To prove that $T$ is indeed a quadratic Leibman sequence, it remains to show that $\delta(i,j)$ is constant for all~$i,j$.
This is done inductively, starting from~\eqref{eq:delta11-T}, which implies that $\delta(1,1)$ is constant.
The details follow.

\begin{lemma}\label{lem:conjug-A-const}
Let $A = \delta(1,1) = a\Bone$ and $B = \delta(1)$.
For $R\in\GZ$, let the \emph{naive degree} $\deg R$ of~$R$ be the sum of the exponents of all occurrences of~$B$ when $R$ is written as a word in the alphabet $A^1,A^{-1},B^1$,~$B^{-1}$.
Then naive degree is invariant under translations:
$\deg R = \deg(\lsub{j}\!R)$ for all $j\in\ZZ$.
Define
\begin{equation*}
  \cc{R} := \lsup{R}\!A.
\end{equation*}
Let $\GAB$ be the subgroup of~$\GZ$ generated by~$A$ and~$B$.
The elements~$\cc{R}$ for $R\in\GAB$ are constant and commute with each other pairwise.
Moreover, $\cc{R}$ depends only on~$\deg R$;
in fact, $\cc{R} = \cc{B^{\deg R}}$.%
\footnote{The naive degree need not be well defined as an integer, but it is well defined as an integer modulo~$N$ if $N$ is the least positive integer (if any) such that $I$ has an expression as a word of naive degree~$N$ (in addition to its expression as the empty word of naive degree~$0$).
  Thus, naive degree induces a well-defined notion of degree (modulo~$N$) with respect to which the identity $\cc{R} = \cc{B^{\deg R}}$ holds.}
\end{lemma}
\begin{proof}
  Let $\cW = \cW(X,Y)$ denote a word in four formal symbols $X,X^{*},Y,Y^{*}$.
  If $P,Q,R\in\GZ$ are such that the word $\cW$ evaluates to~$R$ (using the group operation of~$\GZ$) under the substitutions $X=P$, $X^{*}=P^{-1}$, $Y=Q$, $Y^{*}=Q^{-1}$, we write $R = \cW(P,Q)$.
  For an arbitrary such word~$\cW$, let $R = \cW(A,B)$.
  The equalities $A = \lsub{j}A$ and $\lsub{j}B = A^jB$ imply
  \begin{equation}\label{eq:transl-by-A}
    \lsub{j}\!{R} = \cW(\lsub{j}\!A, \lsub{j}\!B)
    = \cW(A, A^jB) = \cW'(A,B)
  \end{equation}
  for a new word%
\footnote{To be precise, $\cW(A,A^jB)$ is a word $\cW'(A,B)$ where $\cW'(X,Y)$ is obtained from~$\cW(X,Y)$ performing the substitutions $Y=X^jY$, $Y^{*} = Y^{*}X^{-j}$, where powers $X^k$, $X^{-k}$ (for $k\ge 0$) are interpreted as the $k$-words $X\dots X$ and $X^{*}\dots X^{*}$, respectively.}
  $\cW' = \cW'(X,Y)$ using exactly as many of each of the symbols $Y$, $Y^{*}$ as~$\cW$ (but possibly more of~$X$, $X^{*}$).
  It follows that naive degree is invariant under translations.

  We have $\lsub{j}\!\cc{\!R} =  \lsub{j}(RAR^{*}) = \lsub{j}R\cdot\lsub{j}{A}\cdot\lsub{j}{R^{*}} =  \cc{\lsub{j}\!R}$ since $\lsub{j}A = A$.
  Because of the translation invariance of naive degree, once the identity $\cc{R} = \cc{B^{\deg R}}$ is proved, it shall follow that $\cc{R}$ is constant, since $\lsub{j}{\cc{R}} = \cc{\lsub{j}{R}} = \cc{B^{\deg(\lsub{j}R)}} = \cc{B^{\deg R}} = \cc{R}$.

  By identity~\eqref{eq:transl-by-A}, proving $\cc{R}\cc{S} = \cc{S}\cc{R}$ for all $R,S\in\GZ$ reduces to showing $\cc{R}_0\cdot\cc{S}_0 = \cc{S}_0\cdot\cc{R}_0$.
Abusing notation, define $\cc{g} = \lsup{g}{a} = gag^{*}$ for $g\in G$. 
  The remainder of the proof thus reduces to proving
 \begin{enumerate}
 \item $\cc{g}$ depends only on the naive degree $\deg g$ of $g\in G$---defined as the sum of the exponents of~$b$ in an expression $g = \cW(a,b)$ of $g$ as a word~$\cW$ on $a^1, a^{-1}, b^1$, and $b^{-1}$---in fact, $\cc{g} = \cc{b^n}$ where $n = \deg g$,\quad and
 \item the elements $\cc{g} = \lsup{g}\!a$ for $g$ in the subgroup $\Gab$ generated by~$a$ and~$b$ commute pairwise.
 \end{enumerate}
 Since $\cc{g}$ commutes with $\cc{h}$ iff $\cc{g^{*}h}$ commutes with $\cc{I} = a$, property~(2) follows from
 \begin{enumerate}
 \item[(2')] $\cc{g}$ commutes with~$a$ if $g\in\Gab$.
\end{enumerate}
Note that $g\in\Gab$ satisfies properties (1) and~(2') iff either one of $g^{*},ag, a^{*}g$ does.
  
  For $m\ge 0$, let $\Gab_m$ be the set of elements of~$\Gab$ that are words $\cW(a,b)$ using no more than $m$ symbols~$b,b^{*}$. 
  By induction on~$m$, we prove assertions (1) and~(2') for $g\in\Gab_m$.
  (This will prove the assertions for all $g\in\Gab = \bigcup_m\Gab_m$.)
  First, $\Gab_0$ consists of powers~$a^k$ having naive degree zero.
  Since $a$ commutes with~$a^k$, it follows that $\cc{a^k} = \lsup{a^k}{a} = a$; thus, assertions (1) and~(2') hold for $m=0$.
  Next, assume both assertions hold for some fixed~$m\ge 0$.
  Let $g\in\Gab_{m+1}$ be arbitrary.
  Without loss of generality (possibly multiplying $g$ by powers of $a$ or $a^*$ on the left) we may assume that $g = b^{\pm1}h$ with $h\in\Gab_m$.
  If $g=b^{\pm1}h$, then $\deg g = \deg h \pm 1$, and it follows from the inductive hypothesis that $\cc{g} = \cc{b^{\pm1}h} = \lsup{b^{\pm1}}\!{\cc{h}} = \lsup{b^{\pm1}}\!{\cc{b^{\deg h}}} = \cc{b^{\pm1}b^{\deg{h}}} = \cc{b^{\deg g}}$.
  Note that $\cc{b^k} = \lsup{b^k}\!a$ commutes with~$a$ by the hypothesis of Proposition~\ref{prop:Leibman-quadratic};
  hence, so does~$\cc{g}$.
  This shows that assertions (1) and~(2') hold for $m+1$, completing the proof of Lemma~\ref{lem:conjug-A-const}.
  \end{proof}

Continuing the proof of Proposition~\ref{prop:Leibman-quadratic},
let $A = \delta(1,1)$ and $B = \delta(1)$ as above, and let $\ccAB$ be the subgroup of~$\GAB$ generated by elements $\cc{R}$ with $R\in\GAB$.
By Lemma~\ref{lem:conjug-A-const}, $\ccAB$ is an abelian group of constants.
By induction, one shows first that for all~$k,l$ we have $\lsub{k}\delta(l)\in\GAB$, and subsequently that $\delta(k,l)\in\ccAB$ (using Lemma~\ref{lem:conjug-A-const} and the cocycle identity~\eqref{eq:cocycle-gen} to induct on~$l$, then the identity $\delta(i+j,l) = \lsub{j}\delta(i,l)\delta(j,l)$ to induct on~$k$, plus simple manipulations to extend to negative $k,l$).
Thus, $\delta(k,l)$ is constant for all $k,l$.
This implies that $\Bone = \delta(j,k,l) = \Delta^{\!j}\Delta^{\!k}\Delta^{\!l}T$ for all $j,k,l$, showing that $T$ is a quadratic Leibman sequence and concluding the proof of Proposition~\ref{prop:Leibman-quadratic}.

\subsubsection{Some consequences of Proposition~\ref{prop:Leibman-quadratic}}
\label{sec:cor-Leib-quad}

  First, we give some definitions.
  Given a sequence $(x_i:i\in\ZZ)$ in any multiplicative group~$\GG$, there is a natural notion of product $\prod_{i=k}^lx_i$ of the terms $x_i$ ``as $i$ ranges from $k$ to~$l$'';
  it is characterized by the properties
  \begin{enumerate}
  \item $\prod_{i=k}^{k-1}x_i = 1_{\GG}$\quad (the identity of~$\GG$),\quad and
  \item $\prod_{i=k}^{l+1}x_i = \left( \prod_{i=k}^l x_i\right)\cdot x_{l+1}$
  \end{enumerate}
  for all $k,l\in\ZZ$.

  Informally, terms $x_i$ are multiplied left-to-right in succession.
  For fixed~$k$, one obtains the familiar definitions
  \begin{align*}
    \prod_{i=k}^kx_i &= x_k, &
                           \prod_{i=k}^{k+1}x_i &= x_kx_{k+1}, &
    &\dots, &
    &\prod_{i=k}^{k+n}x_i &= x_kx_{k+1}\dots x_{k+n},&
                                                   &\dots,
  \end{align*}
  but also the less familiar
  \begin{align*}
    \prod_{i=k}^{k-2}x_i &= x_{k-1}^{-1}, &
    \prod_{i=k}^{k-3}x_i &= x_{k-1}^{-1}x_{k-2}^{-1}, 
    &\dots,&
             \prod_{i=k}^{k-n-1}x_i &= x_{k-1}^{-1}x_{k-2}^{-1}\dots x_{k-n}^{-1},&
  &\dots.
  \end{align*}
  There is a corresponding notion of product $\oprod$ evaluated in the opposite group~$\GG^{\mathrm{op}}$ of~$\GG$:
  iterated products are computed right-to-left instead, namely
  \begin{enumerate}
  \item $\oprod_{i=k}^{k-1}x_i = 1_{\GG}$,\quad and
  \item $\oprod_{i=k}^{l+1}x_i = x_{l+1}\cdot\left( \oprod_{i=k}^l x_i\right)$
  \end{enumerate}
  for all $k,l\in\ZZ$.%
\footnote{One may alternatively define $\oprod_{i=k}^l x_i := \bigl( \prod_{i=k}^l x_i^{-1} \bigr)^{-1}$.}
  
\begin{proposition}\label{prop:quad-Leib-fmla}
  If $a,b,c$ are elements of a group~$G$ satisfying the commutation relations~\eqref{eq:commut-relns-G}, the unique quadratic Leibman sequence $T$ satisfying $\Delta^{\!1}\Delta^{\!1}T_0 = a$, $\Delta^{\!1}T_0 = b$ and $T_0 = c$ is given by the expression
  \begin{equation}\label{eq:quad-Leib-fmla}
    T_j = \left[\oprod_{i=1}^j(a^{i-1}b)\right]\cdot c
    \qquad\text{for all $j\in\ZZ$.}
  \end{equation}
  In particular, if $a$ and $b$ commute, then $T_j = a^{j\choose 2}b^jc$ where ${j\choose 2} = j(j-1)/2$ for all~$j\in\ZZ$.
\end{proposition}
\begin{proof}
  From the identities $\delta(j+1) = \lsub{j}\delta(1)\cdot\delta(j)$ and  $\lsub{j}\delta(1) = \delta(j,1)\delta(1) = \delta(1,1)^j\cdot\delta(1) = A^jB$, we obtain $\lsub{j}T\cdot T^{*} = \delta(j) = \oprod_{i=1}^j(A^{i-1}B)$.
  Thus, $\lsub{j}T = \oprod_{i=1}^j(A^{i-1}B)\cdot T$.
  Equation~\eqref{eq:quad-Leib-fmla} follows evaluating the latter identity at~$0$.

  If $a$ and $b$ commute, then $T_j = \prod_{i=1}^ja^{i-1}\cdot\prod_{i=1}^jb\cdot c = a^{\sum_{i=1}^j(i-1)}\cdot b^j\cdot c$, where $\sum_{i=1}^j(i-1) = j(j-1)/2 = {j\choose 2}$.
  (For $j\le 0$, the sum $\sum_{i=1}^j$ is understood in the obvious sense analogous to the definition of~$\prod_{i=1}^j$ above.)
\end{proof}

\begin{definition}\label{def:lamplighter-group}
  The restricted wreath product~$\ZwZ$ of~$\ZZ$ with itself is called the \emph{lamplighter group}.
It is realized as a group on generators $(\alpha_k:k\in\ZZ)$ and~$\beta$ subject to the relations
\begin{align}
  \alpha_k\alpha_l &= \alpha_l\alpha_k\label{eq:wreath-commute}\\
  \beta^k\alpha_l &= \alpha_{k+l}\beta^k\label{eq:wreath-shift}\qquad
                    \text{for all $k,l\in\ZZ$.}
\end{align}
\end{definition}

  The subgroups $H = \langle \beta\rangle$ and $K = \langle \alpha_k:k\in\ZZ\rangle$ of~$\ZwZ$ are abelian.
   Together, they generate~$\ZwZ$, and it is easy to show that each is its own centralizer in~$\ZwZ$;
  therefore, $\ZwZ$ has trivial center.
  In particular, $\ZwZ$ is not nilpotent.
  It is, however, solvable, being a semidirect product of the two abelian groups~$H$,~$K$.

\begin{proposition}\label{prop:quad-Leib-wreath}
  Given a quadratic Leibman sequence $T$ in a group~$G$, its discrete differences  $\Delta^{\!i}T, \Delta^{\!i}\Delta^{\!j}T$ ($i,j\in\ZZ$) generate a subgroup $\GZD$ of~$\GZ$ isomorphic to a factor of~$\ZwZ$.
  Actually, the group $\GZD$ is already generated by~$A = \Delta^{\!1}\!\Delta^{\!1}T$ and $B = \Delta^{\!1}T$.
  The values $\Delta^{\!i}T_m, \Delta^{\!i}\Delta^{\!j}T_m$ ($i,j,m\in\ZZ$) of these discrete differences generate a subgroup $\GD$ of~$G$ also isomorphic to a factor of~$\ZwZ$.
  In fact, $a = A_0$ and $b = B_0$ already generate~$\GD$.
Furthermore, there exists a quadratic Leibman sequence such that both $\GZD$ and $\GD$ are isomorphic to~$\ZwZ$ itself.
\end{proposition}
\begin{proof}
  The proof of Proposition~\ref{prop:Leibman-quadratic} shows that $\GZD$ is generated by~$A$ and~$B$.
  \emph{A fortiori,} $\GZD$ is generated by $B$ and $A_{[k]} := \lsup{B^k}\!{A}$ for $k\in\ZZ$ (since $A_{[0]} = A$).
The special case $i=j=l=m=n=1$ (with $k$ arbitrary) of equation~\eqref{eq:commut-relns-GZ} gives the commutation relations
\begin{equation*}
  A \cdot A_{[k]} = A_{[k]} \cdot A,
\end{equation*}
whence the following relations are easily proved using induction and the definition of~$A_{[k]}$:
\begin{align*}
  A_{[k]}A_{[l]} &= A_{[l]}A_{[k]}\\
  B^kA_{[l]} &= A_{[k+l]}B^k \qquad\text{for all $k,l\in\ZZ$.}
\end{align*}
In general, further relations between $A$ and~$B$ may hold;
nevertheless, we see that the generators $A_{[k]}$ ($k\in\ZZ$) and~$B$ of $\GZD$ satisfy the defining relations~\eqref{eq:wreath-commute},~\eqref{eq:wreath-shift} of~$\ZwZ$, so $\GZD$ is isomorphic to a factor of~$\ZwZ$.
  Evaluation at zero is a homomorphism~$\GZ\to G$ that restricts to a surjection $\GZD\to\GD$, so $\GD$ is isomorphic to a factor of~$\GZD$, and thus of~$\ZwZ$, generated by $a = A_0$ and $b = B_0$.

  Reciprocally, let $(\alpha_k:k\in\ZZ)$ and~$\beta$ be the canonical generators of~$\ZwZ$, i.e., these elements obey only relations implied by~\eqref{eq:wreath-commute} and~\eqref{eq:wreath-shift}.
  It follows from Proposition~\ref{prop:Leibman-quadratic} that there is a unique quadratic Leibman sequence $T$ in $\ZwZ$ satisfying $\Delta^{\!1}\Delta^{\!1}T_0 = \alpha_0$, $\Delta^{\!1}T_0 = \beta$ and $T_0 = I$.
  For this sequence~$T$ we have $\GD \simeq \ZwZ$, and hence $\GZD\simeq\ZwZ$ also.
\end{proof}

\begin{corollary}\label{cor:quadr-Leibman-nonnil}
There exists a quadratic Leibman sequence~$T$ with $T_0=I$ whose range generates a non-nilpotent group.
\end{corollary}
\begin{proof}
 The quadratic Leibman sequence~$T$ constructed in the proof of Proposition~\ref{prop:quad-Leib-wreath} has $T_0=I$, and its range generates~$\ZwZ$, which is not nilpotent.
\end{proof}

\begin{remark}
  Bergelson and Leibman used the lamplighter group to construct counterexamples showing that multiple recurrence and multiple convergence results that hold for ergodic actions generating nilpotent groups do fail for non-nilpotent groups~\cite{BergelsonLeibman2004}.
  In contrast to the case of multiple ergodic averages, all (simple) ergodic convergence results in the present manuscript---including Theorems~\ref{thm:PolyMET-Z} and~\ref{thm:PolyMET}---hold under the sole hypothesis that the family~$\Tb$ is a Leibman sequence, which already in the quadratic setting includes cases in which the range of~$\Tb$ is non-nilpotent, per Corollary~\ref{cor:quadr-Leibman-nonnil} above.
\end{remark}

\subsection{Leibman sequences in PET structures}
\label{sec:Leibman-sequence-PET}

\begin{definition}[Discrete difference and abstract Leibman sequence] \label{def:Leibman-seq-PET}
Let $\cM$ be a PET structure over~$\ZZ$.
The \emph{discrete-difference operator} is the function $\Delta^{\!\bullet}:\LZB \to \LZZB$ uniquely characterized by the identity
\begin{equation*}
  (\Delta^{\!\bullet}T)_{i,j} = T_{i+j}\circ T^*_j \qquad\text{for all $i,j\in\cZ$.}
\end{equation*}
Alternatively, for $i\in\cZ$, the left evaluation at~$i$ of $\Delta^{\!\bullet}T$ is $\Delta^{\!i}T = \lsub{i}T\circ T^{*}$.

Let $\Bone = I(\blacksquare)$ denote the constant family $i\mapsto I$ in~$\LZB$. 
Given $d\in\NN$, a \emph{unitary Leibman sequence of degree at most~$d$} is any $T\in\LZB$ satisfying
\begin{equation*}
  \Delta^{\!i_d}\dots\Delta^{\!i_1}\Delta^{\!i_0}T = \Bone\qquad
\text{for all $i_0,i_1,\dots,i_d\in\cZ$,}
\end{equation*}
that takes values in~$\UH$, i.e., also satisfying $T^{*}\circ T = \Bone = T\circ T^{*}$.
A Leibman sequence is a Leibman sequence of any degree~$d$; its degree $\degL\!T$ is the least such~$d$. 
(We define formally $\degL\!\Bone = -\infty$.)
\end{definition}

\begin{remarks}
  \begin{itemize}
  \item Note that abstract Leibman sequences per Definition~\ref{def:Leibman-seq-PET} are ``internal'', i.e., obtained \emph{from elements of~$\LZB$}---that are otherwise only incidentally regarded as functions $\cZ\to\BB$ via evaluation;
    they may be regarded as taking values in the group
    \begin{equation*}
      G = \UU_{\cH} = \{U\in\BB_{\cH} : U\circ U^{*}=I=U^{*}\circ U\}
    \end{equation*}
    of (internal) unitary transformations of the Hilbert space~$\cH$.
    (For this reason, we sometimes refer to these Leibman sequences as \emph{unitary}.)
Accordingly, the defining property of a Leibman sequence amounts to the requirement that $d+1$ discrete-difference operations, \emph{possibly involving nonstandard elements $j\in\cZ$}, always transform $T$ into~$\Bone$, i.e., the constant function $i\mapsto I$ \emph{for all $i\in\cZ$,} not merely for all $i\in\ZZ$.
    The proofs of Theorems~\ref{thm:PolyMET-Z} and~\ref{thm:MetaPolyMET-Z} below crucially depend on the richer structure of~$\cZ$ in saturated PET structures---even if ultimately the results are valid in \emph{all} PET structures, including classical ones whose Leibman sequences are \emph{bona fide} functions $\ZZ\to\UH\subset\BB$.
  \item We have $\Delta^{\!\bullet}T = \lsub{\bullet}T\circ (T^*_{\blacksquare,\cdot})$; 
hence, discrete differentiation is obtained from the $\ZZ$-action on $\LZB$, the right inclusion $\LZB\hookrightarrow\LZZB$, plus the pointwise operations of composition and taking adjoint;
    thus, $\Delta^{\!\bullet}$ may as well be regarded as a distinguished function of any PET structure~$\cM$.
\item The predicate \emph{``\,$T$ is a unitary Leibman sequence of degree at most~$d$''} is captured by a single Henson formula~$\lambda_d(T)$, namely%
\footnote{Expressions of the type $x=y$ such as those in~\eqref{eq:Leibman} are not Henson formulas \emph{sensu stricti}, but may be regarded as abbreviations of formulas $\dd(x,y)\le 0$ (or $\Nrm{y-x}\le 0$ in Banach sorts).}
  \begin{equation}
    \label{eq:Leibman}
    (T\circ T^{*} = \Bone = T^{*}\circ T)
    \wedge
    \forall i_d\dots\forall i_1\forall i_0
    \bigl[\Delta^{\!i_d}(\dots(\Delta^{\!i_1}(\Delta^{\!i_0}T))\dots) = \Bone\bigr].
  \end{equation}
\item Since any group $G$ is realized as a subgroup of a suitable unitary group~$\UH$,%
\footnote{One may identify $G$ with its faithful homomorphic image under the translation action $G\curvearrowright \sL^2(G)$, which realizes~$G$ as a group of unitary transformations of $\cH = \sL^2(G)$.}
it follows that any classical Leibman sequence is realized as a Leibman sequence in a classical PET structure.
In view of Proposition~\ref{prop:quad-Leib-wreath} and Corollary~\ref{cor:quadr-Leibman-nonnil}, we have instances of pointwise ergodic convergence per Theorem~\ref{thm:PolyMET-Z} in the setting of quadratic Leibman sequences $(T_n)$ of unitary operators generating a non-nilpotent group of unitary transformations.
To our knowledge, this is the first explicit example of pointwise convergence of averages of a non-nilpotent (in fact, not even virtually nilpotent%
\footnote{A group is virtually nilpotent if it has a finite-index nilpotent subgroup.})
family of unitary transformations.
\end{itemize}
\end{remarks}

\section{An ergodic theorem for unitary polynomial actions of~$\ZZ$}
\label{sec:poly-PET}

Throughout the end of this section, $\cL$ will be the Henson language for PET structures over~$\ZZ$.
All structures will be in the class~\aPET\ of abstract PET structures over~\ZZ.

\subsection{The sequence of ergodic averages}
\label{sec:avg-seq}

\begin{convention}
  \label{conv:sort-names}
Henceforth, the standalone symbols $\RR,\ZZ,\NN$ shall denote the usual sets of real, integer and natural numbers. 
If $\cM$ is a PET structure over~$\ZZ$, we shall use interpretation of constants (and the density of~$\QQ$ in~$\RR$) to identify $\RR$ with the sort $\RR^{\cM}$, and also $\ZZ$ and $\NN$ with subsets of $\cZ = \ZZ^{\cM}$ and $\cN = \NN^{\cM}$, respectively.%
\footnote{The identification of~$\NN$ with a subset of~$\ZZ$ is neither necessary nor beneficial.
  Theorem~\ref{thm:PolyMET} below considers ergodic averages relative to \Folner\ nets indexed by any countable directed set $\DD$ (in place of~$\NN$) over an arbitrary abelian group~$\GG$ (in place of~$\ZZ$).}
By an abuse of notation, when the structure~$\cM$ is clear from context, we may omit the superscript and write $\cH$, $\BB$, $\AZ$ $\LZ$, \dots\ to denote the sorts $\cH^{\cM}$, $\BB^{\cM}$, $\AZ^{\cM}$, $(\LZ)^{\cM}$, \dots\ of~$\cM$.
\end{convention}

\begin{definition}[Ergodic averages]
\label{def:ergo-avg}
  Let $\cM$ be a PET structure over~$\ZZ$ and let $T\in\LZB$.  
Via the evaluation $\LZB\times\cZ\to\BB$, one may regard $T$ as a function $i\mapsto T_i$.
For $n\in\cN$, the \emph{$n$-th average of~$T$} is
\begin{equation*}
\AV_nT = \pair{T}{\sigma_n}.
\end{equation*}
The \emph{sequence of averages of~$T$} is $\AVb T = (\AV_nT:n\in\NN)$.

Similarly, for $x\in\cH$, the \emph{$n$-th average of~$x$ under~$T$} is $\AV_n\! T(x)$.
The \emph{sequence of averages of~$x$ under~$T$} is $\AVb\! T(x) = (\AV_n\! T(x):n\in\NN)$.  
\end{definition}

We remark that $\ThPET$ ensures the validity of the identities
\begin{equation*}
  \AV_nT = \frac{1}{n+1} \sum_{0\le i\le n} T_i \qquad\text{and}\qquad
  \AV_nT(x) = \frac{1}{n+1} \sum_{0\le i\le n} T_i(x)
\end{equation*}
for (standard) $n\in\NN$;
however, averages as defined above are non-classical if $n\in\cN\setminus\NN$.
On the other hand, the sequence $(\AV_nT : n \in \NN)$ has classical terms, and the study of its convergence is purely classical \emph{a priori}.

\begin{theorem}[Poly-MET/\ZZ: Mean Ergodic Theorem for unitary polynomial actions of~$\ZZ$]
\label{thm:PolyMET-Z}
Let $\cM$ be a PET structure over~$\ZZ$, and let $T\in(\LZ)^{\cM}$ be a Leibman sequence of unitary operators on the Hilbert space~$\cH = \cH^{\cM}$. 
For every $x\in\cH$, the sequence $\AVb T(x) = (\AV_nT(x) : n\in\NN)$ of averages of~$x$ under~$T$ converges in the norm topology of~$\cH$.
\end{theorem}

Theorem~\ref{thm:PolyMET-Z} admits the following uniformly metastable strengthening. 

\begin{theorem}[Metastable Poly-MET/\ZZ]
\label{thm:MetaPolyMET-Z}
Fix~$d\in\NN$. 
There exists a universal metastability rate~$\Eb^d$, depending only on~$d$, that applies uniformly to all sequences $\AVb T(x)$ of averages of arbitrary $x$ in the unit ball of the Hilbert-space sort~$\cH$ under any Leibman sequence $T$ in~$\UH$ of degree at most~$d$ in any PET structure $\cM$ over~$\ZZ$.
\end{theorem}

The rest of this section is devoted to proving Theorems~\ref{thm:PolyMET-Z} and~\ref{thm:MetaPolyMET-Z}.

\subsection{Proof preliminaries}
\label{sec:lemmas}

\begin{lemma}[Dominated Convergence Theorem in PET structures]
\label{thm:DCT-PET}
Let $\cL$ be the language of PET structures.
Let~$\cM$ be any \emph{saturated} PET structure.
Let $\fb = (\varphi_n\colon n\in\NN)$ be a bounded sequence in~$\LZH$. 
For all $x\in\ZZ^{\cM}$ assume that the sequence $\fb(x) = (\varphi_n(x) : n\in\NN)$ in~$\cH^{\cM}$ is convergent.
Then, for arbitrary $\mu\in\MM^{\cM}$, the sequence $\pair{\fb}{\mu} = \big(\pair{\varphi_n}{\mu} : n\in\NN)$ in $\cH^{\cM}$ is convergent.
\end{lemma}

(We will only require the special case of Lemma~\ref{thm:DCT-PET} in which $\varphi$ is of the form $n\mapsto\AV_n\! T$ with $T\in\LZB$.) 

\begin{proof}
  A saturated PET structure $\cM$ is a \emph{Banach integration framework} (with Banach sort~$\cH^{\cM}$ and measure-space sort~$\ZZ^{\cM}$) as defined in Appendix~\ref{sec:Banach-int-framewk}.
  Thus, Lemma~\ref{thm:DCT-PET} follows from Theorem~\ref{thm:DCT} whose statement and proof are in Appendix~\ref{sec:DCT}.
\end{proof}
To state the next lemma we need a definition. 
Let the \emph{reverse difference} operator $\Nab:\LZB \to \LZZB$ be the mapping $T\mapsto\Nab T$ characterized by the property that $\Nab T$ evaluates to the function $(i,j)\mapsto
T_j\circ T^{*}_{i+j}$. 
We write $\nabla^iT$ to denote the left evaluation of~$\Nab T$, i.e., $\nabla^iT$ evaluates to the mapping $j\mapsto T_j\circ T^{*}_{i+j}$.
Note that $\nabla^iT$ is the translate by~$i$ of the (forward) difference of $T$ with step~$-i$, i.e., $\nabla^iT =  \lsub{i}\Delta^{\![-i]}T$ holds for all~$i\in\cZ$.
Just like the forward difference operator~$\Delta^{\!\bullet}$, the reverse difference operator $\Nab$ is explicitly definable in any PET structure since it is obtained by composing functions of the structure (the $\ZZ$-action $\LZB\to\LZZB$, the shear map on $\LZZB$, the pointwise adjoint operation $\LZB\to\LZB$, the left inclusion $\LZB\hookrightarrow\LZZB$, and the pointwise composition~$\LZZB\times\LZZB\to\LZZB$);
thus, $\Nab$ may as well be considered a distinguished function of any PET structure.

\begin{lemma}\label{lem:convol}
Let $\cM$ be a PET structure such that $\cN = \NN^{\cM}$ contains a nonstandard natural number $M\in\cN\setminus\NN$. 
  For every standard natural $n\in\NN$ and $T\in\LZB$:
  \begin{equation*}
    (\AV_n T)\circ(\AV_M T)^{*}
    = \pair{\Pair{\Nab T}{\sigma_n}}{\sigma_M}.
  \end{equation*}
\end{lemma}
In less cryptic notation, the equation above reads:
\begin{equation*}
  (\AV_n T)\circ(\AV_M T)^{*}
  = \int\!\! \AV_n(\nabla^i T)\,d\sigma_M(i).
  \end{equation*}
  Every step of the proof below is justified by an axiom of~$\ThPET$.
  We prefer to use informal integral notation to make the argument transparent.
    \begin{proof}
Note that $(\AV_{\!M}T)^{*} = \AV_{\!M}(T^{*})$ (in fact, $\pair{T}{\mu}^{*} = \pair{T^{*}}{\mu}$ for all $\mu\in\MM$).
For $m,n\in\NN$ we have:
\begin{equation*}
   \begin{split}
    (\AV_nT)\circ(\AV_mT)^{*} 
&= 
\int T_j\,d\sigma_n(j)
\circ
\int T^{*}_i\,d\sigma_m(i)
=
\iint
T_j\circ T^*_i\,d\sigma_m(i)\,d\sigma_n(j) \\
&=
\iint T_j\circ T^*_{i+j}\,d\sigma_m(i+j)\,d\sigma_n(j)\\
&\qquad
\text{(letting $i = i+j$ in the inner integral)} \\
&=
\iint\nabla^i T_j
\,d\sigma_m(i+j)\,d\sigma_n(j) \\
&=
\iint\nabla^i T_j
\,d\sigma_n(j) \,d\sigma_m(i)\\
&\qquad -\iint 
\nabla^i T_j\,
d[\sigma_m- \lsub{j}(\sigma_m)](i)\,d\sigma_n(j),
\end{split}
\end{equation*}
where $\lsub{j}(\sigma_m)$ is the translation of $\sigma_m$ by~$j$. 
Since $\Nrm{\nabla^iT_j} = \Nrm{\lsub{i}\Delta^{\![-i]}T_j} \le \Nrm{T}^2$ for all $i,j$:
\begin{equation*}
  \begin{split}
    \Nrm{(\AV_nT)\circ(\AV_mT)^* 
- \int\!\! \AV_n(\nabla^iT)\,d\sigma_m(i)}
\le 
\Nrm{T}^2 \cdot
\max_{0\le j\le m} \|\sigma_m - \lsub{j}(\sigma_m)\|.
  \end{split}
\end{equation*}

Given fixed $\epsilon>0$ and $n\in\NN$, let $m_{\epsilon,n}$ be the smallest natural number satisfying $m_{\epsilon,n}\ge 2n/\epsilon$. 
Clearly, $\|\sigma_m - \lsub{j}(\sigma_m)\| \le 2n/(m+1) \le \epsilon$ for $j\le n\le m+1$.
(For $m$ large and $n$ small, this inequality captures the ``approximate invariance'' of the long interval $\{0,1,\dots,m\}$ of~$\ZZ$ under small translations, i.e., the \Folner\ property of the collection of such intervals.)
Then we have
\begin{equation}\label{eq:descent}
  \begin{split}
\Nrm{(\AV_nT)\circ(\AV_mT)^*
- \int\!\! \AV_n(\nabla^i T) \,d\sigma_m(i)}
\le \epsilon\|T\|^2
\qquad\text{whenever $m\ge m_{\epsilon,n}$.}
  \end{split}
\end{equation}
Since $M\in\cN\setminus\NN$ satisfies $M\ge m_{\epsilon,n}$ for all $n\in\NN$ and $\epsilon>0$, the assertion in Lemma~\ref{lem:convol} follows.
\end{proof}

We offer some remarks on the proof of Lemma~\ref{lem:convol} above, which is the crux of our approach to proving Theorem~\ref{thm:PolyMET-Z}. 
Despite its rather short length, it sheds light on the various sorts and distinguished functions in PET structures.
There are no double integrals as such but rather iterated integrals $\LZZB\to\LZB$ and $\LZB\to\BB$---the order of the integration (first on the left and then on the right variable, or vice versa) is immaterial as ensured by the PET axiom
\begin{equation*}
  (\forall{\cT}\in\LZZB)(\forall\mu,\nu\in\MM)
  \bigl[
  \pair{\Pair{\mu}{\cT}}{\nu} = \pair{\Pair{\cT}{\nu}}{\mu}
  \bigr].
\end{equation*}
The validity of the substitution $i=i+j$ in the inner integral is justified by the compatibility of the shear transformation on $\LZZB$ and the action of~$\ZZ$ on~$\MM$:
\begin{equation*}
  (\forall{\cT}\in\LZZB)(\forall\mu\in\MM)
  \bigl[\Pair{\cT}{\mu} = \Pair{\widetilde{\cT}}{\lsub{\bullet}{\mu}}\bigr].
\end{equation*}
Other steps in the proof admit similar formal justifications by axioms of PET structures.

\begin{lemma}\label{lem:divergence}
  Let $\cM$ be a \emph{saturated} Henson structure with an ordered sort $(\cN,\le)$ extending $(\NN,\le)$, and let $\fb$ be a sequence explicitly defined by a $\cL[S]$-term $\varphi(\mathtt{n})$ of sort~$s$  (i.e., $\fb$ is the sequence $(\varphi(n):n\in\NN)$ in~$\cM$, where $\mathtt{n}$ is a variable of sort~$\cN$ and $S$ is some set of parameters of the universe of~$\cM$).
Then every sub-sequential limit of $\varphi$ is of the form $\varphi(M)$ for some $M\in\cN\setminus\NN$. 
If $\varphi(M)=\varphi(N)$ for all~$M\in\cN\setminus\NN$, then the sequence $\fb = (\varphi(n):n\in\NN)$ converges. 
In such case, the common value $\varphi(M)$ is the limit $\lim_{n\to\infty}\varphi(n)$.
\end{lemma}
The proof of Lemma~\ref{lem:divergence} is a routine application of saturation left to the reader.

\begin{lemma}\label{lem:density}
Let~$\cX$, $\cY$ be metric spaces, and let $\varphi : (x,n)\mapsto\varphi_n(x)$ be a function from $\cX\times\NN$ to~$\cY$ such that $\varphi_n(\cdot):\cX\to\cY$ is $1$-Lipschitz for each $n\in\NN$ (i.e., $d(\varphi_n(x),\varphi_n(y))\le d(x,y)$ for $x,y\in\cX$). 
Let $S$ be a dense subset of~$\cX$ such that $\fb(x)$ converges for all $x\in S$.  
Then $\fb(x)$ converges for all~$x\in\cX$.
\end{lemma}
We omit the straightforward proof of Lemma~\ref{lem:density}.

\subsection{Proof of Theorem~\ref{thm:PolyMET-Z}}
\label{sec:proofPolyMET-Z}

For each fixed Leibman degree $d\in\NN$, we first prove Theorem~\ref{thm:PolyMET-Z} for unitary Leibman sequences of degree at most~$d$ in any \emph{saturated} PET structure~$\cM$.
The descent argument on the degree~$d$ is characteristic of Bergelson's PET induction~\cite{Bergelson:1987}.

The assertion is trivial for $T = \Bone$.
If $T$ (is pointwise unitary and) has Leibman degree $\degL T=0$, we have $T_i\circ T^*_0 = \Delta^{\!i}T_0 = I$, hence $T_i = T_0$ for all $i\in\cZ$, so $T$ is constant. 
Thus, the sequences $\AVb T$ and $\AVb T(x)$ are also constant (all terms are equal to~$T_0$ and $T_0(x)$, respectively), so Theorem~\ref{thm:PolyMET-Z} follows for Leibman polynomials of degree~$0$.

Assume now that the assertion in Theorem~\ref{thm:PolyMET-Z} is proved for all $T\in\LZB$ having Leibman degree less than some positive integer~$d$. 
Fix $T\in\LZB$ with $\degL(T)=d$.

\begin{lemma}\label{lem:structured}
  If $x = (\AV_{\!M}\!T)^*(y)$ for some $M\in\cN\setminus\NN$ and $y\in\cH$, then $\AVb T(x)$ converges.
\end{lemma}
\begin{proof}
  Note that $\AVb T(x)$ is bounded by~$\Nrm{T}\Nrm{x}$.
  By Lemma~\ref{lem:convol}, we have $\AV_nT(x) = \AV_nT\circ\AV_{\!M}T^{*}(y) = \pair{\Pair{\Nab T(y)}{\sigma_n}}{\sigma_M}$ for $n\in\NN$.
  For $i\in\cZ$ we have $\nabla^iT = \lsub{i}\Delta^{\![-i]}T$.
  By the invariance of Leibman degree under translation and the assumption $\degL(T)\le d$, we have $\degL(\nabla^iT) = \degL(\lsub{i}\Delta^{\![-i]}T) < d$, and hence $\AVb(\nabla^iT)(y)$ is convergent for all $i\in\cZ$ by the inductive hypothesis. 
An application of Lemma~\ref{thm:DCT-PET} concludes the proof.
\end{proof}

The space $\Struct$ of \emph{structured elements} of~$\cH$ (relative to~$T$) is the closure of the linear span of all elements of the form $(\AV_{\!M}\!T)^{*}(y)$ for $M\in\cN\setminus\NN$ and $y\in\cH$.
By linearity and Lemmas~\ref{lem:density} \&~\ref{lem:structured}, $\AVb T(x)$ converges for all structured elements~$x$.
(The $1$-Lipschitz condition follows from the inequalities $\|\!\AV_n\!T(x) - \AV_n\!T(y)\| \le \|\!\AV_n\!T\|\|y-x\|$ and $\|\!\AV_n\!T\| = \|\!\left\langle T,\sigma_n \right\rangle\!\| \le \|T\|\cdot\|\sigma_n\| = 1\cdot1 = 1$.)

The space $\Rand$ of \emph{pseudorandom elements} of~$\cH$ (relative to~$T$) is the orthogonal complement of $\Struct$ in~$\cH$.
By the fundamental theorem of linear algebra and the definition of structured elements, an element $x\in\cH$ is pseudorandom precisely when $\AV_{\!M}\!T(x)=0$ for all $M\in\cN\setminus\NN$. 
By Lemma~\ref{lem:divergence}, $\AVb T(x)$ converges to zero in this case.

Combining the pseudorandom and structured cases using linearity, we deduce that all averages $\AVb\!T(x)$ converge.  
This concludes the inductive step of the proof of Theorem~\ref{thm:PolyMET-Z} in any saturated PET structure~$\cM$.

To conclude the proof for any PET structure over~$\ZZ$, let $\cM$ be any (not necessarily saturated) PET structure, and let $\tM$ be a saturated elementary $\cL$-extension of~$\cM$ in Henson's logic. 
For fixed degree $d\in\NN$ and $T\in(\LZB)^{\cM}$, the property that $T$ is a unitary Leibman sequence of degree at most~$d$ is $\cL$-axiomatizable, hence it is true in~$\tM$ when $T$ is regarded as an element of~$(\LZB)^{\tM}$.
For $x\in\cH^{\tM}$ we have proved that $\AVb\! T(x)$ converges since $\degL(T)\le d$.  
\emph{A fortiori}, $\AVb\! T(x)$ converges for~$x\in\cH^{\cM}$. 
This concludes the proof of Theorem~\ref{thm:PolyMET-Z} in full generality.

\subsection{Proof of Theorem~\ref{thm:MetaPolyMET-Z}}
\label{sec:proofMetaPolyMET-Z}

Let $\tL$ expand the language $\cL$ of PET structures with new constants $(\mathtt{T}, \mathtt{x}, \mathtt{y}_n : n<\omega)$, with  $\mathtt{T}$ of sort~$\LZB$, and $\mathtt{x}$ and all $\mathtt{y}_n$ of sort~$\cH$.
For fixed $d\in\NN$, consider the $\tL$-theory
\begin{equation*}
  \Lambda_d = 
  {\ThPET} \cup
  \left\{
    \lambda_d(\mathtt{T}),
    \|\mathtt{x}\|\le 1,
    \mathtt{y}_n = \AV_n\!\mathtt{T}(\mathtt{x}) :
    n<\omega
  \right\},
\end{equation*}
where $\lambda_d(\mathtt{T})$ is formula~\eqref{eq:Leibman} stating that the interpretation of $\mathtt{T}$ is Leibman of degree at most~$d$.
Note that $\Lambda_d$ is a uniform theory: $\lambda_d(\mathtt{T})$ implies $\|\mathtt{T}\| \le 1$, hence also $\Nrm{\texttt{y}_n}\le 1$.
Every model~$\tM$ of~$\Lambda_d$ is an expansion of a PET structure~$\cM$ having the form~$\tM = (\cM,T,x,\AVb T(x))$.
By Theorem~\ref{thm:PolyMET-Z}, all sequences $(c_n^{\tM}) = \AVb T(x)$ are convergent.
An application of Proposition~\ref{thm:UMP} finishes the proof of Theorem~\ref{thm:MetaPolyMET-Z}.

\subsection{Proof of Theorem~\ref{thm:AbelPolyMET}}
\label{sec:proofAbelianMET-Z}

A classical Leibman sequence $\Tb = (T_k : k\in\ZZ)$ with $\degL(T)\le d$ in an \emph{abelian} subgroup $K$ of the group~$\UH$ of unitary operators on a Hilbert space~$\cH$ is easily shown to have the form $T_k = U_0\circ U_1^k\circ U_2^{k\choose 2}\circ\dots\circ U_d^{k\choose d}$ where $U_j = \Delta^jT_0$ and ${k\choose j} = k(k-1)\dots(k-j+1)/j!$ are binomial coefficients for $j=0,1,\dots,d$.
Since the functions $k\mapsto{k\choose j}$ ($j=0,1,\dots,d$) are a $\ZZ$-basis for polynomial mappings $p:\ZZ\to\ZZ$ of degree at most~$d$, Theorem~\ref{thm:AbelPolyMET} is an immediate corollary of Theorems~\ref{thm:PolyMET-Z} and~\ref{thm:MetaPolyMET-Z}.

\section{A Mean Ergodic Theorem for unitary polynomial actions of abelian groups}
\label{sec:general}

To formulate our most general result on convergence of averages, we replace $\NN$ with an arbitrary countable directed set~$(\DD,\le)$ and $\ZZ$ with an arbitrary abelian group~$(\GG,+)$ endowed with a countable \Folner\ $\DD$-net $\Fb = (\cF_j:j\in\DD)$ of nonempty finite subsets of~$\GG$. 
These assumptions are sufficient to ensure that the proofs of natural generalizations of Theorems~\ref{thm:PolyMET-Z} and~\ref{thm:MetaPolyMET-Z} carry through in this more general context, \emph{mutatis mutandis,} from those given in Section~\ref{sec:poly-PET}.

\begin{theorem}[Poly-MET: Mean Ergodic Theorem for unitary polynomial actions of an abelian group]
\label{thm:PolyMET}
Fix an abelian group~$(\GG,+)$ and a \Folner\ net $\Fb = (\cF_j:j\in\DD)$ of subsets of~$\GG$, indexed by a countable directed set $(\DD,\le)$.
Let $\cH$ be a Hilbert space.
Let $T : \GG\to\UH$ be a polynomial mapping, in Leibman's sense, into the group~$\UH$ of unitary transformations of~$\cH$.
For every $x\in\cH$, the $\DD$-net $\AVb T(x) = (\AV_iT(x) : i\in\DD)$ in~$\cH$, of averages relative to~$\Fb$:
\begin{equation*}
  \AV_iT(x) = \frac{1}{\#\cF_i} \sum_{g\in\cF_i}T_g(x)
\end{equation*}
of~$x$ under~$T$, converges in the norm topology of~$\cH$.

In fact, given fixed choices of $\GG$, $\DD$, $\Fb$ and $d\in\NN$, there exists a rate of metastability
\begin{equation*}
  \Eb^{(\GG,\DD,\Fb,d)} =
  \left(E_{\epsilon,\eta} : \epsilon>0,
  \eta\in\prod_{i\in\DD}\Pfin(\DD_{\ge i})\right)
\end{equation*}
(with $E_{\epsilon,\eta}\in\Pfin(\DD)$ for each $\epsilon,\eta$)
that applies universally to all sequences $\AVb T(x)$ for any element $x$ in the unit ball of any Hilbert space~$\cH$ and any Leibman polynomial $T:\DD\to\UH$ of degree at most~$d$.
\end{theorem}

\begin{remark}
The definition of Leibman polynomial mapping $T:\GG\to\UH$ is a straightforward generalization of that of Leibman sequence $\ZZ\to\UH$~\cite{Leibman:2002}.
The discrete difference $\Delta^{\!g}T$ of $T$ with step $g\in\GG$ is the mapping $h\mapsto T_{g+h}\circ T^{*}_h$.
Define $\degL(T)\le 0$ if $\Delta^gT = \Bone$ (where $\Bone$ is the constant mapping $g\mapsto I$).
Recursively, let $\degL(T)\le d+1$ mean that $\degL(\Delta^{\!g}T)\le d$ for all $g\in\GG$.
Then $T$ is a Leibman mapping if $\degL(T)\le d$ for some~$d$;
the least such~$d$ is $\degL(T)$, although we adopt the convention $\degL(\Bone) = -\infty$.
\end{remark}

We only provide an outline of the proof of Theorem~\ref{thm:PolyMET} since it is formally identical to the arguments in Sections~\ref{sec:lemmas}--\ref{sec:proofMetaPolyMET-Z}.
The definition of classical PET structure over~$\GG$ is completely analogous to that of PET structure over~$\ZZ$ in section~\ref{sec:PET-classical}---simply replace all instances of~$\ZZ$ by~$\GG$ and those of~$\NN$ by~$\DD$.
The \Folner\ sequence $\Fb$ is captured indirectly via the \Folner\ measure map $\sigma : \DD\to\MM$, where 
\begin{equation*}
  \sigma_i = \frac{1}{\#\cF_i} \sum_{g\in\cF_i}\delta_g\qquad
  \text{for all $i\in\DD$.}
\end{equation*}

The Henson language~$\cL$ for the class of classical PET structures over~$\GG$ is clear.
Any model of the theory $\ThPET^{\DD,\GG}$ of such classical structures is an (abstract) PET structure over~$\GG$.
(Note that the language~$\cL$ depends on both~$\GG$ and~$\DD$;
the theory $\ThPET^{\DD,\GG}$ further depends on the choice of the \Folner\ net~$\Fb$.)

Analogues of Theorems~\ref{thm:PolyMET-Z} and~\ref{thm:MetaPolyMET-Z} hold in the class of all PET structures over~$\GG$.
The scheme of proof is exactly the same.
The countability hypothesis on~$\DD$ is an essential hypothesis in Theorem~\ref{thm:DCT}, which enters the proof via an analogue of Lemma~\ref{thm:DCT-PET}.

Lemma~\ref{lem:convol} uses the exact same definition of reverse difference $(\nabla^gT)_h = T_h\circ T^{*}_{g+h}$.
Its proof is adapted using the definition of \Folner\ net as we now indicate:
By definition, given $\epsilon>0$ and $g\in\GG$ there is $k = k_{g,\epsilon}\in\DD$ such that the symmetric difference $\cF_l\triangle(g+\cF_l)$ has cardinality at most $\epsilon\cdot\#\cF_l$ for all $l\ge k$.
Thus, letting $K_{i,\epsilon} = \max\{k_{g,\epsilon} : g\in\cF_i\}$, we have $\|\sigma_j - \lsub{g}(\sigma_j)\| \le \epsilon$ for all $g\in\cF_i$ provided $j\ge K_{i,\epsilon}$.
Whence follows the proof of an analog of Lemma~\ref{lem:convol} stating that $(\AV_jT)\circ(AV_KT)^{*} = \int\AV_{\!j}(\nabla^gT)\,d\sigma_K(g)$ holds whenever $j\in\DD$ and $K\in\DD^{\cM}$ satisfies $K\ge i$ for all $i\in\DD$.

Lemma~\ref{lem:divergence} continues to hold provided one replaces the nonstandard natural numbers $M,N$ with nonstandard elements $J,K$ of~$\DD^{\cM}$ that satisfy $J,K\ge i$ for all standard elements $i\in\DD$.

The arguments in Sections~\ref{sec:proofPolyMET-Z} and~\ref{sec:proofMetaPolyMET-Z} apply verbatim once the lemmas in Section~\ref{sec:lemmas} have been adapted, completing the proof of Theorem~\ref{thm:PolyMET}.

\appendix

\section{A Dominated Convergence Theorem for notions of integration in Banach spaces}
\label{sec:appendix}

This appendix bears a close connection to our prior manuscript on measure, integration and metastable convergence in Henson structures~\cite{Duenez-Iovino:2017}. 
Our main goal is proving Lemma~\ref{thm:DCT-PET}. 
Rather than doing so in the specific context of PET structures, we prove a more general result (Theorem~\ref{thm:DCT}) about sequences of integrals of functions on a finite measure space taking values in a Banach space.
This requires a number of preliminary steps.

\subsection{Integration structures}
\label{sec:int-struct}
We recall the class of \emph{integration structures} (with underlying finite positive measure) introduced in our earlier manuscript, to which we refer the reader for details~\cite{Duenez-Iovino:2017}. 
These are saturated models of the Henson theory $\Th_{\int}$ of integration with respect to a positive finite measure on structures with (classical) sorts $\RR$, $\Omega$, $\AO$, $\LO$ where $\RR$ is the set of real numbers, $\AO$ is a $\sigma$-algebra of subsets of~$\Omega$, and $\LO$ is the set of bounded $\AO$-measurable (everywhere-defined) real functions on~$\Omega$. 
(Here $\Omega$, $\AO$ are discrete while $\RR$, $\LO$ are real Banach spaces.) 
This theory contains all Henson formulas involving the functions and distinguished constants below that are valid in such structures:
\begin{itemize}
\item \emph{Constants:} Rational numbers $r\in\RR$, zero vector in the Banach sort $\LO$, an arbitrary point (``anchor'') $\omega_0\in\Omega$, the empty set $\emptyset\in\AO$, and the improper subset $\Omega\in\AO$.
\item \emph{Functions:}
  \begin{itemize}
  \item Arithmetic operations (addition and multiplication), absolute value and lattice operations (binary min and max) on~$\RR$;
  \item The characteristic function $\bin{\cdot}{\cdot}: \Omega\times\AO\to\{0,1\}\subseteq\RR$ of the membership relation $\in$ on~$\Omega\times\AO$;
  \item Banach operations (addition, scalar multiplication) and norm on $\LO$ 
(namely, $\|f\| = \sup_{x\in\Omega}|f(x)|$ for $f\in\LO$---note that an almost-everywhere null function~$f$ has positive norm per this definition unless $f=0$ everywhere);
\item The evaluation map $\LO\times\Omega\to\RR : (f,x)\mapsto f(x)$;
  \item The Banach lattice operations (binary min and max) on $\LO$;
  \item The unary operation of pointwise absolute value $f\mapsto\nrm{f}$ on~$\LO$ where $\nrm{f}\in\LO$ is the function $x\mapsto |f(x)|$;
  \item The Boolean algebra operations of union, intersection and relative complement $S\mapsto S^{\complement} = \Omega\setminus S$ on~$\AO$;
  \item The characteristic-function map $\chi : \AO\to\LO : S\mapsto\chi_S$;
  \item A positive finite measure $\mu$ on~$\Omega$;
  \item The integration operator $I : \LO\to\RR : f\mapsto\int_{\Omega}f\,d\mu$.
  \end{itemize}
\end{itemize}
Let $\cL$ be any Henson language including sort symbols $\RR,\Omega,\AO,\LO$ as well as constant and function symbols matching the lists above, and let $\Th_{\int}$ be the $\cL$-theory of such structures $\cM = (\mathbf{S}, \mathbf{F}, \mathbf{C})$, where $\mathbf{S}$ is the list of sorts, $\mathbf{F}$ the collection of distinguished functions, and $\mathbf{C}$ the set of distinguished elements of~$\cM$. 
An \emph{(abstract) pre-integration structure} is a model of~$\Th_{\int}$. 
An \emph{integration structure} is a saturated model of~$\Th_{\int}$. 
If $\cM$ is any pre-integration structure (whether saturated or not), then via interpretation of constants, the membership relation $\bin{\cdot}{\cdot}$, and the evaluation map, we may identify $\RR^{\cM}$ with $\RR$, ${\AOM}$ with a Boolean algebra of subsets of~$\Omega^{\cM}$, and $\LO$ with a set of functions~$\Omega\to\RR$.
However, $\AO$ need \emph{not} be a $\sigma$-algebra.
Accordingly, $\mu^{\cM}$ is typically just a \emph{finitely} (not countably) additive measure on~$(\Omega^{\cM},{\AOM})$, while elements $f\in\LO$ are identified with uniformly bounded functions on~$\Omega$ that may only be \emph{approximately} $\AO$-measurable.%
\footnote{A function $f$ on~$\Omega$ is \emph{approximately $\AO$-measurable} if for all rational $r<s$ there exists $A\in\AO$ such that $f^{-1}((-\infty,r))\subseteq A\subseteq f^{-1}((-\infty,s])$---a property axiomatizable by countably many Henson formulas in the logic of approximate satisfaction (\cite{Duenez-Iovino:2017}, Proposition~4.4).} 
Nevertheless, in earlier work we have shown how the classical (i.e., $\sigma$-additive) theory of integration of bounded measurable functions over a finite measure space and the corresponding version of the Dominated Convergence Theorem are recovered essentially verbatim in \emph{saturated} Henson integration structures via an analogous construction to that of Loeb measure in nonstandard analysis~\cite{Duenez-Iovino:2017}.

\subsection{Loeb structures}
\label{sec:Loeb}

\begin{definition}[Loeb structure]\label{def:Loeb-struct}
  Let $\ThL$ be the reduct of the $\cL$-theory $\Th_{\int}$ of integration structures with a positive measure to the language~$\cL'$ obtained by removing from~$\cL$ the symbol for sort $\LO$ as well as all functions and constants involving~$\LO$ (such as the symbol~$I$ for the integral).
  A model of $\ThL$ is a \emph{pre-Loeb structure.} 
(Note that $\cM$ may be an $\tL$-structure for a language $\tL$ properly extending the language~$\cL'$ of Loeb structures, and thus have other sorts, functions and constants prescribed by~$\tL$ but not by~$\cL'$.) 

A \emph{Loeb structure} is a saturated pre-Loeb structure.
\end{definition}
Note that, for the present discussion, we are requiring the measure~$\mu^{\cM}$ in a pre-Loeb structure~$\cM$ to be \emph{positive}.

If $\cM$ is any pre-Loeb structure, the \emph{set underlying} a given $A\in{\AOM}$ is
\begin{equation*}
  [A] = \{x\in\Omega^{\cM} : \bin{x}{A} = 1\}.
\end{equation*}
We may (externally) identify~$A$ with~$[A]$ since
\begin{equation*}
  (\forall A)(\forall B)(A=B \leftrightarrow (\forall x)(\bin{x}{A} = \bin{x}{B}))
\end{equation*}
is a sentence in~$\ThL$.%
\footnote{Although Henson's languages have no conditional connective ``$\rightarrow$'', when $P$ is a discrete predicate (i.e., a term taking only the values $0,1$), a non-Henson formula such as $P(x)\rightarrow \varphi(x)$ can be semantically identified with the Henson formula $(P(x)\le 1/2)\vee \varphi(x)$.
(By contrast, the converse $\varphi(x)\rightarrow P(x)$ is not semantically equivalent to a Henson formula in general.)
  When both $P,Q$ are discrete, a biconditional $P(x)\leftrightarrow Q(x)$ can similarly be rewritten as a Henson formula.
The assertion ``\emph{$R$ is discrete}'' is captured by the Henson formula $(\forall x)(R(x)=0 \vee R(x)=1)$, where  ``$R(x)=r$'' is itself an abbreviation for ``$(R(x)\le r)\wedge(R(x)\ge r)$''.}

\begin{definition}[Loeb measure and Loeb-measurable sets]\label{def:mu-meas}
Let $\cM$ be a pre-Loeb structure with positive measure $\mu = \mu^{\cM}$.

A set $S\subseteq\Omega^{\cM}$ is \emph{$\AO$-measurable} (or just \emph{measurable}) if $S = [A]$ for some~$A\in{\AOM}$
(i.e., if ``$S\in{\AOM}$''---modulo the identification of~$S=[A]$ with $A$ itself).

A set $S\subseteq\Omega^{\cM}$ is \emph{$\mu$-measurable} (or \emph{Loeb-measurable (modulo~$\mu$)}) if for every $\epsilon>0$ there exist measurable~$A,B\in{\AOM}$ such that $[A]\subseteq S\subseteq [B]$
and $\mu(B-A)\le\epsilon$.

The \emph{Loeb measure} of a Loeb-measurable set~$S$ is
\begin{equation*}
  \muL(S)
  = \sup\{\mu(A) : A\in\AO^{\cM}, [A]\subseteq S\}
  = \inf\{\mu(B) : B\in\AO^{\cM}, [B]\supseteq S\}.
\end{equation*}

The \emph{Loeb algebra} of~$\cM$ is the collection~$\Amu$ of all Loeb-measurable subsets of~$\Omega^{\cM}$. 
\end{definition}

Note that $\Amu$ is an \emph{external} collection of subsets of~$\Omega^{\cM}$. 
It depends on~$\cM$ and has no intrinsic definition otherwise.
It is easy to check that $\Amu$ is an algebra of sets (i.e., closed under finite unions and intersections as well as complements). 
In fact, as soon as $\cM$ is at least $\omega$-saturated (i.e., types over a countable set of parameters are realized), $\Amu$ is a $\sigma$-algebra that is complete for $\muL$ in the sense that any subset of a $\muL$-null set is itself $\muL$-null (\cite{Duenez-Iovino:2017}, Proposition~3.4).
On the other hand, no degree of saturation ensures that~$\Amu$ is closed under unions of subfamilies of size~$\omega_1$ or more.

\subsection{Integration frameworks}
\label{sec:int-framewks}

We need to introduce the notion of \emph{(real) integration framework}, which generalizes integration structures as presented in section~\ref{sec:int-struct}. 
Roughly speaking, an integration framework is a saturated model of the theory of the operations of integration with respect to arbitrary finite (positive or signed) measures on a measure space.

Consider the reduct~$\tM$ of a \emph{classical} pre-integration structure~$\cM$, obtained by removing from~$\cM$ the distinguished measure~$\mu$ and all the functions involving~$\mu$ (including the integration operator~$I$).
Now expand $\tM$ to a structure $\cM'$ with a new Banach sort $\MO$ containing all finite (signed, real-valued) measures $\mu$ on~$\Omega$ plus the following distinguished functions and constants:
\begin{itemize}
\item \emph{Constants:} The zero measure $0\in\MO$;
\item \emph{Functions:}
  \begin{itemize}
  \item Vector space operations of addition and scalar multiplication on~$\MO$.
  \item Banach norm of total variation on~$\MO$:\\
 $\|\mu\| = \sup\{A\in\AO : |\mu(A)|+|\mu(A^{\complement})|\}$;
  \item The inclusion maps:
    \begin{itemize}
    \item $\Omega\hookrightarrow\AO : x\mapsto\{x\}$,
    \item $\Omega\hookrightarrow\MO : x\mapsto\delta_x$ (the unit point mass at~$x$);
    \end{itemize}
  \item The evaluation map $\MO\times\AO\to\RR : (\mu,A)\mapsto \mu(A)$ (which is 1-Lipschitz by definition of the norm $\Nrm{\cdot}$ on~$\MO$);
  \item The \emph{total variation} map $\LO\to\LO$: $\mu\mapsto|\mu|$ where $|\mu|$ is the (positive) measure of total variation of~$\mu$:
$|\mu|(A) = \sup\{|\mu(A\cap B)| + |\mu(A\cap B^{\complement})| : B\in\AO \}$;
  \item The integration operator $\pair{\cdot}{\cdot} : \LO\times\MO\to\RR : f\mapsto \pair{f}{\mu} = \int_{\Omega}f\,d\mu$.
  \end{itemize}
\end{itemize}
Given a language $\cL$ for pre-integration structures, let $\cL'$ be obtained from $\cL$ by removing the symbols $\boldsymbol{\mu}, I$ for the distinguished measure and integral operator, and adding a new sort symbol $\MO$ as well as new constant and function symbols per the list above. 

\begin{definition}[Real integration framework]\label{def:int-framewk}
  A \emph{classical (real) pre-integration framework} is any $\cL'$-structure $\cM'$ as described above.  
An \emph{(abstract) real pre-integration framework} is any model of the Henson theory~$\ThIntR$ of classical real pre-integration frameworks.%
\footnote{It is straightforward to verify that $\ThIntR$ is a uniform theory.}  

More generally, any structure $\cM$ in a language expanding~$\cL'$ such that the $\cL'$-reduct of~$\cM$ is a pre-integration framework in the above sense will be called a pre-integration framework.

A \emph{(real) integration framework} is a saturated real pre-integration framework.
\end{definition}

\subsection{Banach integration frameworks}
\label{sec:Banach-int-framewk}

There is no completely general notion of integration of functions taking values in an arbitrary Banach space~$\BB$---not even for bounded functions $F:\Omega\to\BB$ on a finite measure space~$(\Omega,\AO)$. 
However, it is very natural to require that any such notion of Banach integration should build upon the classical integral of real-valued functions. 
Our viewpoint is that any reasonable notion of Banach integration must expand a (pre-)integration framework to a \emph{Banach (pre-)integration framework} per Definition~\ref{def:Banach-int-fwk} below.

Consider expansions $\cM' = (\mathbf{S}', \mathbf{F}', \mathbf{C}')$ of real integration frameworks~$\cM = (\mathbf{S}, \mathbf{F}, \mathbf{C})$ where $\mathbf{S}'\supset \mathbf{S}$ contains two new sorts $\BB$ and $\LOB$, while $\mathbf{F}'\supset \mathbf{F}$ and $\mathbf{C}'\supset \mathbf{C}$ contain new functions and symbols as follows:
\begin{itemize}
\item Addition, scalar product and norm on~$\BB$ making it a real Banach space.
\item Addition, scalar product and norm on~$\LOB$ making it a Banach space.
\item The zero elements of~$\BB$ and $\LOB$.
\item An evaluation map $\LOB\times\Omega\to\BB : (F,x)\mapsto F(x)$ such that $\|F\| = \sup\{\|F(x)\| : x\in\Omega\}$.
(Thus, elements of $\LOB$ may be identified with functions $\Omega\to\BB$.)
\item The operation of multiplication $\LO\times\LOB\to\LOB : (f,F)\mapsto fF$ (such that $(fF)(x) = f(x)F(x)$ for all $x\in\Omega$) under which $\LOB$ is an $\LO$-module.
\item The inclusion map $\BB\to\LOB : T \mapsto T(\blacksquare)$ where $T(\blacksquare)\in\LOB$ is identified via evaluation with the constant function $\Omega\to\BB : x\mapsto T$.
\item The pointwise-norm map $\nrm{\cdot} : \LOB\to\LO$ satisfying $\nrm{F}(x) = \Nrm{F(x)}$ for all $F\in\LOB$, $x\in\Omega$.
\item An operation of Banach integration, namely a pairing $\Pair{\cdot}{\cdot} : \LOB\times\MO\to\BB$ satisfying the following properties:
\begin{enumerate}
\item $\llangle\cdot,\cdot\rrangle$ is bilinear;
\item $\llangle\cdot,\cdot\rrangle$ is compatible with the integration~$\langle\cdot,\cdot\rangle$ of real functions:
  \begin{enumerate}
  \item For all $T\in\BB$ and $f\in\LO$: $\llangle fT,\mu\rrangle = \langle f,\mu\rangle T$.
  \item For all $F\in\LOB$: $\|\llangle F,\mu\rrangle\| \le |\langle|F|,|\mu|\rangle|$.
  \end{enumerate}
\end{enumerate}
\end{itemize}

\begin{definition}[Banach integration framework]\label{def:Banach-int-fwk}
A language $\cL'$ expanding the language~$\cL$ of real pre-integration frameworks with the new sort symbols plus symbols for the functions and constants above is called a \emph{language for Banach integration frameworks.}

Let $\ThIntR$ be the Henson $\cL$-theory of real pre-integration frameworks, and let $\ThIntB$ extend $\ThIntR$ with further Henson $\cL'$-axioms capturing the properties of new sorts, functions and constants stated above (in semantically equivalent terms, let $\ThIntB$ be the Henson $\cL'$-theory of those expansions $\cM'$ of real pre-integration frameworks having the properties above).%
\footnote{The verification that $\ThIntB$ is a uniform theory is routine.}

A \emph{Banach pre-integration framework} is a model of~$\ThIntB$.
More generally, if $\tL$ is a language extending~$\cL'$ and $\cM$ is an $\tL$-structure whose reduct $\cM\restriction\cL'$ is a model of $\ThIntB$, we shall still call $\cM$ a Banach pre-integration framework.

A \emph{Banach integration framework} is a saturated model of~$\ThIntB$.
\end{definition}

\begin{remark}\label{rem:Banach-expansion}
The question whether a real pre-integration framework $\cM$ admits an expansion to a Banach pre-integration framework~$\cM'$ is very delicate. 
In general, the answer may be negative. 
However, when $\Omega^{\cM}$ is a finite set the answer is affirmative: 
It suffices to let $(\LOB)^{\cM'}$ be the set of all functions $F : \Omega^{\cM}\to\BB^{\cM}$, and also let $\Pair{F}{\mu} = \sum_{x\in\Omega^{\cM}}F(x)\mu(\{x\})$. 
The remaining ingredients of the expansion are defined in the obvious manner.
Similarly, an expansion $\cM'$ also exists if the Banach sort~$\BB^{\cM}$ has finite dimension (using a basis for~$\BB^{\cM}$, real-valued integration extends to $\BB^{\cM}$-valued integration in the straightforward classical fashion).
\end{remark}

\subsection[DCT in Banach integration frameworks]{A Dominated Convergence Theorem for nets of functions in Banach integration frameworks}
\label{sec:DCT}

In order to formulate a version of the Dominated Convergence Theorem~\ref{thm:DCT} for  integration frameworks below, we fix a directed set $(\DD,\preceq)$ so we can eventually discuss convergence of nets on it.%
\footnote{I.e., $\preceq$ is a nonstrict partial order on~$\DD$ such that any two $i,j\in\DD$ have an upper bound~$k$.} 
Classical sequences indexed by the directed set~$(\NN,\le)$ of natural numbers are of particular interest.
Only \emph{infinite} directed sets are useful as tools to define and study notions of convergence in analysis and topology;
on the other hand, critical results such as Theorem~\ref{thm:DCT} depend on the \emph{countability} of the directed set, so we may as well fix an infinite countable directed set~$(\DD,\preceq)$ for the remainder of the manuscript (this hypothesis will be made explicit whenever needed).

\begin{definition}\label{def:D-net}
  Fix a directed set~$(\DD\preceq)$.
  For $i\in\DD$, the \emph{final segment of~$\DD$ starting at~$i$} is $\DD_{\succeq i} = \{j\in\DD : j\succeq i\}$ (i.e., the set of elements equal to or greater than~$i$ in~$\DD$).
  A $\DD$-net $\ab$ in a metric space $(X,\dd)$ is any function $\DD\to X : i\mapsto a_i$.
  The \emph{spread of~$\ab$ from~$i$} is
  \begin{equation*}
   \spr_{\succeq i}(\ab) = \sup_{j,k\succeq i} \dd(a_j,a_k).
 \end{equation*}
  The \emph{oscillation} of~$\ab$ is
  \begin{equation*}
   \osc(\ab) = \inf_{i\in\DD}\spr_{\succeq i}(\ab).
 \end{equation*}
 The net~$\ab$ \emph{converges} if $\osc(\ab)=0$.
\end{definition}

\begin{theorem}[Dominated Convergence Theorem in Banach integration frameworks]
  \label{thm:DCT}
  Fix a \emph{countable} directed set~$\DD$.
Let~$\cM$ be any (saturated) Banach integration framework. 
Let $\fb$ be a bounded $\DD$-net in~$(\LOB)^{\cM}$.
For every $x\in\Omega^{\cM}$ and $\mu\in\MO^{\cM}$, let $\fb(x)$ denote the net~$(\varphi_j(x) : j\in\DD)$ and $\Pair{\fb}{\mu}$ the net $\big(\Pair{\varphi_j}{\mu} : j\in\DD)$ in~$\BB^{\cM}$.
Then we have
\begin{equation*}
  \osc(\Pair{\fb}{\mu})
  \le
  \Nrm{\mu}
  \sup_{x\in\Omega^{\cM}}\osc(\fb(x)).
\end{equation*}
In particular, if the net $\fb(x)$ is convergent for all $x\in\Omega^{\cM}$, then $\Pair{\fb}{\mu}$ is convergent.
\end{theorem}

The proof of Theorem~\ref{thm:DCT} below is an adaptation of our earlier one for real-valued notions of integration (\cite{Duenez-Iovino:2017}, Proposition~5.3).

  Recall that a collection $\cF$ of subsets of a set~$U$ is a \emph{(proper) filter on~$S$} if (\emph{i})~$\emptyset\notin\cF$, (\emph{ii})~$\cF$ is closed under finite intersections, and (\emph{iii})~$\cF$ is upward closed: if $A\in\cF$ and $A\subseteq B\subseteq U$, then $B\in\cF$. 
  A proper filter $\cF$ is an \emph{ultrafilter} if $A\in\cF$ or $U\setminus A\in\cF$ for all $A\subseteq U$.

  For an introduction to ultrafilters, ultralimits and ultraproduct constructions in model theory, the reader is referred to Bell and Slomson's monograph~\cite{BellSlomson2006}.
  
\begin{definition}\label{def:full-filter}
  If $X$ is any nonempty set, let $\Pfin(X)$ be the family of finite nonempty subsets of~$X$.
  We call a filter $\cF$ on $\Pfin(X)$ \emph{greedy} if it contains all the sets $X_{\supseteq S} = \{T\in\Pfin(X): T\supseteq S\}$ for all $S\in\Pfin(X)$.
\end{definition}

Note that the collection $\{X_{\supseteq S} : S\in\Pfin(X)\}$ is a filter base on $\Pfin(S)$ since $X_{\supseteq S}\cap X_{\supseteq T} = X_{\supseteq S\cup T}$. 
(This means that the collection of subsets of $\Pfin(X)$ that are supersets of $X_{\supseteq S}$ for some $S\in\Pfin(X)$ is a filter on~$\Pfin(X)$.)%
\footnote{Recall that a \emph{filter base} on a set~$U$ is a collection $\mathcal{E}$ of subsets of~$U$ such that $\emptyset\notin \mathcal{E}$ and $\mathcal{E}$ is downward directed by inclusion in the sense that if $A,B\in \mathcal{E}$ then $A\cap B \supseteq C$ for some $C\in \mathcal{E}$.
The \emph{filter~$\cF$ with base $\mathcal{E}$} is the collection of all subsets of~$U$ that are supersets of some $A\in \mathcal{E}$.}
By a routine application of the axiom of choice, greedy ultrafilters on~$\Pfin(X)$ exist whenever $X$ is nonempty.
Observe that the principal ultrafilter generated by a fixed $S\in\Pfin(X)$ is greedy precisely when $S=X$; 
thus, if $X$ is infinite, greedy ultrafilters on $\Pfin(X)$ are nonprincipal.

\begin{lemma}\label{lem:ext-L-meas-approx}
  Let $f : \Omega^{\cM}\to\RR$ be bounded and (externally) $\muL$-measurable.
  Then there exists $\widetilde{f}\in(\LO)^{\cM}$ such that $f(x) = \widetilde{f}(x)$ for $\muL$-almost all~$x\in\Omega^{\cM}$ and $\inf_xf(x)\le \widetilde{f} \le \sup_xf(x)$.
\end{lemma}
\begin{proof}
  Let $f$ be a $\muL$-measurable bounded external function on~$\Omega^{\cM}$, and let $a = \inf_xf$, $b = \sup_xf$.
The assertion is trivial if $a=b$ or if $\Nrm{\mu}=0$---just take a constant $\widetilde{f}$ in $[a,b]$.
  Otherwise, we have $a<b$ and, replacing $\mu$ with $\nrm{\mu}$, we may assume $\mu$ to be a positive measure without loss of generality. 
  By definition of Loeb measurability, for rational~$r\in[a,b]$ and integer $n\ge 1$ there exist $A^r_n,B^r_n\in\AO^{\cM}$ such that $[A_n^r]\subseteq \{f\le r\}$, $[B_n^r]\subseteq \{f\ge r\}$, and $\muL\{f\le r\} - \mu(A_n^r) \le 1/n$, $\muL\{f\ge r\} - \mu(B_n^r) \le 1/n$.
  For fixed~$r$, the sequences $(A^r_n)$, $(B^r_n)$ may be constructed recursively to ensure $A_m^r\subseteq A^r_n$ and $B^r_m\subseteq B^r_n$ for $m\le n$.
We may also assume $A_m^r=1_{\AO}$ if $r\ge b$, and $B^s_m=1_{\AO}$ if $s\le a$.

Let $f^r_n = a\cdot(1-\chi_{B^r_n}) + r\cdot\chi_{B^r_n}$ and $g^r_n = b\cdot(1-\chi_{A^r_n}) + r\cdot\chi_{A^r_n}$.
  Let $Q$ be the set of rational numbers in~$[a,b]$.
The construction of $(A^r_n)$ and $(B^r_n)$ implies that $f^r_m \le f^r_n \le f \le g^r_n\le g^r_m$ for all $r\in Q$ and $m\le n$.
For $I\in\Pfin(Q)$ of cardinality~$n$, let $f^I = \max\{f^r_n:r\in I\}$, $g^I = \min\{g^r_n:r\in I\}$.
  Observe that $f^r_n\le f^I\le f^J\le f\le g^J\le g^I\le g^r_n$ if $I\subseteq J$, $r\in I$ and $\card(I)\ge n$.
  Since $a<b$ by assumption, $Q$ is infinite countable.
  Let $\cU$ be a greedy ultrafilter on $\Pfin(Q)$.
  By saturation, there are $\widetilde{f}, \widetilde{g}\in(\LO)^{\cM}$ realizing the $\cU$-ultralimit of the types $\tp_S(f^I,g^I)$ over the set of parameters $S = \{\mu,a,b\}\cup \{A^r_n,B^r_n\}_{r\in Q,n\in\NN^{*}}$.
  From the construction of~$\cU$ as a greedy ultrafilter, the definition of ultralimit, and the meaning of realization of a type, it is easy to verify that
\begin{equation*}
 a \le f^r_n \le \widetilde{f}\le \widetilde{g} \le g^r_n\le b \quad\text{for all $r\in Q$ and $n\ge 1$.}
\end{equation*}
It follows that for fixed $r\in Q$ we have $S^r := \bigcup_n [A^r_n] =  \bigcup_n\{g^r_n\le r\} \subseteq \{\widetilde{g}\le r\}$.
  On the other hand, by construction of $A^r_n$ we have $S^r \subseteq \{f\le r\}$ and $\muL(S^r) = \sup_n\mu(A^r_n) = \muL\{f\le r\}$.
Thus, $S^r\subseteq \{f\le r\}\cap\{\widetilde{g}\le r\}$ and $\muL(S^r) = \muL\{f\le r\}$;
hence, $\{f\le r\}$ is $\muL$-almost included in~$\{\widetilde{g}\le r\}$.
By a completely analogous argument, $\{\widetilde{f}\ge r\}$ \muL-almost includes $\{f\ge r\}$.
These almost-inclusions for every (rational) $r\in Q$ are easily shown to imply the \muL-a.e.\ inequalities $\widetilde{g}\le f\le \widetilde{f}$.
However, $\widetilde{f}\le \widetilde{g}$, so in fact $\widetilde{f} = f = \widetilde{g}$ (\muL-a.e.)
\end{proof}

\begin{lemma}\label{lem:sup-inf-internal}
  Fix $\mu\in\MO^{\cM}$ and let $\ab = (a_i:i<\omega)$ be a sequence of external $\muL$-measurable functions $\Omega^{\cM}\to\RR$.
Then there exist $\sigma,\iota\in(\LO)^{\cM}$ such that $\sigma(x) = \sup_{i<\omega}a_i(x)$ and $\iota(x) = \inf_{i<\omega}a_i(x)$ for $\muL$-almost all~$x\in\Omega^{\cM}$.

If the (external) sequence~$\ab$ consists of internal functions, i.e., it is a sequence in~$(\LO)^{\cM}$, then $\sigma,\iota$ may be chosen so $\sigma\ge\sup_{i<\omega} a_i$ and $\iota\le\inf_{i<\omega}a_i$.
\end{lemma}
\begin{proof}
  It is routine to show that $f=\sup_{i<\omega}a_i$ and $g=\inf_{i<\omega}a_i$ are $\muL$-measurable, so the first assertion follows from Lemma~\ref{lem:ext-L-meas-approx}.
  
  When $\ab$ is a sequence of internal functions,  let $\cU$ be any nonprincipal ultrafilter on $\omega$ and let $\sigma$ realize the $\cU$-ultralimit of the types $\tp_S(a^k)$ over the set of parameters $S = \{\mu\}\cup \{a_i\}_{i<\omega}$, where $a^k = \max\{a_i:i\le k\}$.
  (Recall that $\LO$ is endowed with the binary lattice operation $\max\{a,b\}$, which trivially defines $n$-ary maximum operations for all $n\ge 1$.)
  The verification that $\sigma$ has the required properties is routine.
  The construction of $\iota$ is identical upon replacing ``max'' by ``min''.
\end{proof}

\begin{proof}[Proof of Theorem~\ref{thm:DCT}]
  The asserted inequality evidently holds if $\mu=0$.
  Otherwise, using a Jordan decomposition $\mu = \mu_+-\mu_-$ where $\mu_+ = (\nrm{\mu}+\mu)/2$ and $\mu_- = (\nrm{\mu}-\mu)/2$ are positive, the proof is easily reduced to the case in which $\mu$ is a probability measure, which we assume henceforth.

  Choose $C$ such that $\Nrm{\varphi_i}\le C$ for all $i$.
  For $j,k\in\DD$ let $\varphi^{j,k} = \nrm{\varphi_k-\varphi_j} \in (\LO)^{\cM}$.
  Since $\DD$ is countable, Lemma~\ref{lem:sup-inf-internal} implies that for each $i\in\DD$ there is $\sigma^i\in(\LO)^{\cM}$ with $\|\sigma^i\|\le 2C$ such that $\sigma^i$ is $\muL$-a.e.\ equal to $\spr_{\succeq i}\fb = \sup_{j,k\succeq i}\varphi^{j,k}$.
  Similarly, $\osc\fb = \inf_i \spr_{\succeq i}\fb$ is $\muL$-a.e.\ equal to $\inf_i\sigma^i$, hence to some $\omega\in(\LO)^{\cM}$ with $\Nrm{\omega}\le 2C$.

  Let $s = \sup_x\osc(\fb(x))$ and fix $t>s$.
  Since $\{x\colon \osc\fb(x)\ge t\}$ is empty (by choice of~$s$ and $t$) and $\omega(x)=\osc(\fb(x))$ for $\muL$-a.e.~$x$, the set $\{x\colon \omega(x)\ge t\}$ is \muL-null.
  Since $\omega = \inf_i\spr_{\succeq i}\fb$ (\muL-a.e.), we have $\inf_i\muL\{x\colon\spr_{\succeq i}\fb(x) \ge t\} = 0$.
  Thus, for arbitrary fixed $\epsilon>0$ we have $\muL\{x\colon\spr_{\succeq i}\fb(x) \ge t\} \le \epsilon$ for some~$i = i_\epsilon\in\DD$.
  (This depends crucially on the hypothesis that $\DD$ is countable.)
  It follows that for $j,k\succeq i$:
  \begin{equation*}
    \begin{split}
      \nrm{\Pair{\varphi_k-\varphi_j}{\mu}} &\le \Pair{\nrm{\varphi_k-\varphi_j}}{\mu}
      = \int \nrm{\varphi_k(x)-\varphi_j(x)}d\muL(x)\\
      &= \left(\int_{\{x\colon \spr_{\succeq i}\fb(x) < t\}} + \int_{\{x\colon \spr_{\succeq i}\fb(x) \ge t\}}\right) \nrm{\varphi_k(x)-\varphi_j(x)}d\muL(x)\\
      &\le t\cdot\muL\{x\colon \spr_{\succeq i}\fb(x) < t\}
      + 2C\epsilon \le t + 2C\epsilon.
  \end{split}
  \end{equation*}
  This proves that $\spr_{\succeq i}\Pair{\fb}{\mu} \le t + 2C\epsilon$.
  As $t>s$ and $\epsilon>0$ are arbitrary, $\osc\Pair{\fb}{\mu} \le s$.
\end{proof}

\subsection{A Uniform Metastability Principle for nets in Henson structures}
\label{sec:UMP}

\begin{proposition}[Uniform Metastability Principle (UMP)]\label{thm:UMP}
Fix a directed set~$(\DD,\preceq)$.  
Fix a Henson language~$\cL$ including constants $(a_j : j\in\DD)$ all of a common sort~$\BS$.
Let $\cT$ be a uniform~\cL-theory such that for every model $\cM$ of~$\cT$ the net $\ab^{\cM} = (a_j^{\cM} : j\in\DD)$ is convergent.
Then there exists a metastability rate $\Eb = \Eb^{\cT}$ depending only on~$\cT$ that applies uniformly to all sequences $\ab^{\cM}$ in all models $\cM$ of~$\cT$.
\end{proposition}
\begin{proof}
(\cite{Duenez-Iovino:2017}, Proposition~2.4.)
  Assume no such rate of metastability exists.
  Then there exist $\epsilon>0$ and a sampling $\eta \in \prod_{i\in\DD}\Pfin(\DD_{\succeq i})$ such that for every $S\in\Pfin(\DD)$ there is a model $\cM = \cM^S_{\epsilon,\eta}$ of $\cT$ such that $\ab = \ab^{\cM}$ satisfies $\epsilon\le \spr_{\eta_i}(\ab) = \max\{\dd(a_j,a_k) : j,k\in\eta_i\}$ for all $i\in S$.
  By the compactness theorem for Henson logic, there is a model $\cM$ of $\cT$ such that $\ab = \ab^{\cM}$ satisfies $\spr_{\eta_i}(\ab) \ge \epsilon$ for all $i\in\DD$, and hence $\osc(\ab)\ge\epsilon$, contradicting the hypothesis that $\ab^{\cM}$ converges.
\end{proof}

\def\cprime{$'$}
\providecommand{\bysame}{\leavevmode\hbox to3em{\hrulefill}\thinspace}
\providecommand{\MR}{\relax\ifhmode\unskip\space\fi MR }
\providecommand{\MRhref}[2]{%
  \href{http://www.ams.org/mathscinet-getitem?mr=#1}{#2}
}
\providecommand{\href}[2]{#2}

\end{document}